\documentclass[11 pt, letterpaper, oneside]{article}
\usepackage[utf8]{inputenc}
\usepackage{inputenc,xcolor,amsthm,hyperref, cite, color, url,amsfonts,amsmath,amssymb,caption,multicol,graphicx,fancyhdr,tikz,pgfplots,subcaption,comment}
\usepgfplotslibrary{groupplots}
\usepackage[toc,page]{appendix}
\usepackage[bottom=3cm, headsep=1cm,headheight=2cm]{geometry}

\DeclareMathOperator{\E}{\mathbb{E}}
\DeclareMathOperator{\Pb}{\mathbb{P}}

\DeclareMathOperator{\R}{\mathbb{R}}
\DeclareMathOperator{\Var}{Var}

\renewcommand{\a}{\alpha}
\renewcommand{\b}{\beta}
\renewcommand{\l}{\lambda}

\newcommand{\g}{\gamma}

\newcommand{\s}{\sigma}
\newcommand{\z}{\zeta}
\newcommand{\De}{\Delta(\epsilon)}

\newcommand{\e}{\epsilon}
\newcommand{\sq}{\sqrt}
\newcommand{\f}[2]{\frac{#1}{#2}}
\newcommand{\eqd}{\stackrel{\textrm{d}}{=}}

\usepackage{blindtext}
\usepackage{soul}

\theoremstyle{theorem}
\newtheorem{theorem}{Theorem}[section]
\newtheorem{lemma}[theorem]{Lemma}
\newtheorem{proposition}[theorem]{Proposition}
\newtheorem{corollary}[theorem]{Corollary}

\theoremstyle{definition}

\newtheorem{remark}[theorem]{Remark}
\newtheorem{example}[theorem]{Example}

\newcommand{\pushright}[1]{\ifmeasuring@#1\else\omit\hfill$\displaystyle#1$\fi\ignorespaces}

\usepackage[affil-it]{authblk}
\title{A Tandem Fluid Network with L\'evy Input\\ in Heavy Traffic}
\author[1]{D.T. Koops}
\author[2]{O.J. Boxma}
\author[1]{M.R.H. Mandjes}
\affil[1]{Korteweg-de Vries Institute, University of Amsterdam}
\affil[2]{Eurandom and Department of Mathematics and Computer Science, Eindhoven University of Technology}
\date{\today}

\newcommand{\vb}{\vspace{3mm}}

\begin{document}
\maketitle

\begin{abstract}\noindent
In this paper we study the stationary workload distribution of a fluid tandem queue in heavy traffic. We consider different types of L\'evy input, covering compound Poisson, $\a$-stable L\'evy motion (with $1<\a<2$), and Brownian motion. In our analysis we separately deal with   
L\'evy input processes with increments that  have finite and infinite variance. A distinguishing feature of this paper is that we do not only consider the usual heavy-traffic regime, in which the load at one of the nodes goes to unity, but also a regime in which we {\it simultaneously} let the load of both servers tend to one, which,  as it turns out, leads to entirely different heavy-traffic asymptotics. Numerical experiments indicate that under specific conditions the resulting simultaneous heavy-traffic approximation significantly outperforms the usual heavy-traffic approximation.  
\end{abstract}

%\cite{Kella1993}
\section{Introduction}
In this paper we study a fluid tandem queue {that consists of two servers in series}. A continuous-time stochastic input process, namely spectrally-positive L\'evy input, {feeds into} the first queue (also: upstream queue). The first server empties the upstream queue  at a deterministic rate $r_1$, immediately feeding the second (also: downstream) queue. The downstream server leaks at some deterministic rate {$r_2$; to make the system non-trivial we throughout assume $r_2<r_1$}. We are interested in the stationary workloads in both queues  {in  heavy-traffic regimes that we specify below}.

The heavy-traffic regime was first considered in \cite{Kingman}: one lets the load of the system tend to one, while simultaneously scaling the workload {in such a way} that a non-degenerate limiting distribution is obtained. Kingman's approach was mainly based on manipulating Laplace-Stieltjes transforms; {this approach we also follow in our paper}. {Another approach relies on the  functional central limit theorem in combination with the continuous mapping theorem, see e.g.\   \cite{pro}}. In \cite{shwa} both approaches are compared, and {the traditional heavy-traffic results, which assume the increments of the input process to have a finite variance, are generalized to the infinite-variance case}. For {excellent surveys} we refer to \cite{gly} and the book \cite{whwa}.

Tandem queueing systems in which both queues are experiencing heavy-traffic conditions have been studied before. \cite{HAR} has {focused on the classical setting} of a GI/G/1-type tandem in which discrete entities (`customers') arrive to the first queue, leave this queue after their service has been completed, and then join a second queue, in which they undergo service as well. In {such} queueing systems the correlation between both queues is typically {\it negative}, as the first queue being relatively large could be a consequence of long service times in that queue, which in turn correspond to long inter-arrival times in the second queue, and hence a relative small number of customers in the second queue. Harrison manages to quantify the resulting (negative) covariance between the populations in both queues in heavy traffic.  Importantly, in the fluid setting considered in our work this {reasoning} does {\it not} hold. {More specifically, for the types of models we study, the correlation between both workloads is {\it positive}: in our setting  large workloads in the upstream queue likely correspond to  large workloads in the downstream queue.}

\vb

Fluid tandem queues with spectrally-positive L\'evy input have been studied before, see for example \cite{kelwhi} and \cite{Kella1993}. The results concerning the joint distribution of the steady state of the workloads were generalized to a {more general class of queueing networks} in e.g.\ \cite{debdierol}. These results play an important role for our analysis, and are therefore summarized in Section \ref{app:LST}. A more extensive account of L\'evy-driven networks can be found {in Chapter 12 and 13 of \cite{Mandjes}. }

\vb 

The load of a server is defined as the average input rate into the server divided by its service rate. The load can thus be increased by increasing the average input rate, or lowering the service rate. In case of a single-node system, both methods are equivalent in the sense that they lead to the same heavy-traffic results. However, for multi-node systems (such as tandem queues), increasing the average input to the first server only leads to heavy traffic in the downstream server (recall that $r_1>r_2$). To be more general we therefore adapt the {\it service rates} appropriately, while keeping the input process fixed. Taking this approach opens up the possibility that the servers experience heavy traffic simultaneously. In this paper we study both types of heavy traffic, and refer to them as follows:
\begin{itemize}
\item Regime I: Only the downstream server has a load that tends to unity ({whereas the first queue does not operate under heavy traffic});
\item Regime II: The up- and downstream server have loads that {\it simultaneously} tend to unity.
\end{itemize}

\vb

{In general terms, the results we find for Regime I are much in line with those for heavy traffic in single queues, whereas for Regime II we obtain limiting distributions which, to the best of our knowledge, have not appeared before. More specifically, our contributions are:}
\begin{itemize}
\item
For Regime I we find that the steady-state distribution of the workload in the second queue is similar to the one of the first queue. Moreover,  the up- and downstream workloads are asymptotically independent in the heavy-traffic limit. 
\item
{In Regime II, we establish the interesting feature that the workloads do not decouple in heavy traffic, i.e., some dependence between the up- and downstream workloads remains}. Moreover, the marginal steady-state distribution of the downstream queue is crucially different from the one obtained in Regime I. {This has practical implications: as  verified through a set of experiments, Regime II approximations tend to outperform those based on Regime I, particularly} when the load of both servers is large.
\end{itemize}

\vb

We find that, as in the single-server case, there is a dichotomy between {input processes that have increments with finite and infinite variance}; {as a consequence, they have to be dealt with separately}. {We have derived Regime I results in both cases, and for the case of finite variance
we have also succeeded in addressing the technically more demanding Regime II.}

{In Regime I we prove that the stationary workload of the downstream queue has an exponential distribution (for the case of finite variance) or Mittag-Leffler distribution (for infinite variance). Remarkably, the same distributions (up to some factor) were found for single queues;} apparently, the fact that there is a server that modifies the input process  hardly affects the limiting distribution. {In addition, similar results were also found for waiting times in the corresponding GI/G/1 queues; see e.g.\ \cite{Boxma1999} for the case of infinite variance.}

\vb

The paper is organized as follows. In Section \ref{sec:Levy queues} we introduce our framework of queueing models with L\'evy input;  we subsequently explain the fluid L\'evy tandem queueing model that we consider, and recall  results that play a key role throughout the paper. As mentioned above,  there is a dichotomy between the case of finite (Section \ref{sec:finvar}) and infinite variance (Section \ref{sec:infvar}).   
In Section \ref{sec:finvar}  we first consider Brownian input, for which all computations can be done explicitly, 
and then turn to  general spectrally-positive L\'evy input; the section also present numerical experiments that indicate that the Regime II approximation typically outperforms the Regime I approximation. Section \ref{sec:infvar}, which focuses on infinite variance input, covers results for compound Poisson input and $\a$-stable input. Finally, in Appendix \ref{app:rv} we state Tauberian theorems, that are used in Section \ref{sec:infvar}.

\section{L\'evy driven queues}
\label{sec:Levy queues}
Let $(\Omega,\mathcal{F},\mathbb{F},\Pb)$ be a filtered probability space, where $\mathcal{F}=\mathcal{F}_T$ and the filtration $\mathbb{F}=\{\mathcal{F}_t:t\geq0\}$ satisfies the usual conditions. Let $T\in[0,\infty]$ denote the time horizon, which is allowed to be infinite. We start by providing a short introduction to single-node L\'evy-driven queues; for more details, see e.g.\ the exposition in \cite{Mandjes}.

To define single-node L\'evy-driven queues, we first switch to a discrete-time setting. Let $\{Q_n, n\in\mathbb{N}\}$ denote the workload process, where $Q_n$ is the workload at the beginning of the $n$-th timeslot. Let $Y_n$ denote the net input in the queue during the $n$-th timeslot. Then the discrete-time queue can be described through the well-known Lindley equation
\[
Q_{n+1} = \max\{ Q_n + Y_n, 0\}.
\]
By iterating this recursion and defining $X_n := \sum_{i=0}^n Y_i$, this eventually leads to 
\begin{equation}
\label{eq:discr}
Q_n = X_n + \max\left\{x, \max_{0\leq i\leq n} -X_i\right\},\quad n\in\mathbb{N},
\end{equation}
with initial workload $Q_0=x\geq0$. 
{Single-node L\'evy-driven queues, denoted by $\{Q_t, t\in\mathbb{R}_{\geq0}\}$, can be seen as the continuous-time analogue of Equation \eqref{eq:discr}:}
\[
Q_t = X_t + \max\{x,L_t\}, \quad t\geq0,
\]
with 
\[
L_t := \sup_{0\leq s \leq t} - X_s = - \inf_{0\leq s \leq t} X_s.
\]
%Assuming the queue has been running since from $t=-\infty$, one can alternatively write 
%\[
%Q_t = \sup_{s\leq t}(X_t-X_s).
%\]
%To ensure the existence of a stationary distribution, we need the assumption that $\E %X_1 <0$. For L\'evy processes it can be shown that for the stationary distribution %$Q$
%\[
%Q \eqd \sup_{t\geq0} X_t, 
%\]
%which is commonly attributed to Reich. 

\subsection{A fluid tandem queueing model}
\label{sec:model}

Suppose we have a L\'evy driven fluid tandem queue consisting of two servers. The input process $J=\{J_t, t\geq0\}$ feeds the first server (upstream server). The workload from the first server then flows continuously, at a fixed rate $r_1$, to the second server (downstream server). The downstream server empties itself at a fixed rate $r_2$. Consider Figure \ref{fig:diagram} for a diagram of this model and consider Figure \ref{fig:sample} for a typical sample path when the arrival process is a renewal process. {Assume $r_2<r_1$, as otherwise the second queue would remain empty}.  We  use two different parametrizations in this paper. In Regime I we parametrize 
\[
r_1=\E J_1 + r,\quad \text{for some fixed $r>0$, and $r_2 = \E J_1 + \e.$}
\]
For Regime II, we take
\[
r_1=\E J_1 + \g\e \quad\text{and}\quad r_2=\E J_1 + \e,\quad\text{in which $\g>1$ to guarantee that $r_1>r_2$.}
\]

\begin{figure}
    \begin{center}
        \begin{tikzpicture}[>=latex]
            % the rectangle with vertical rules
            \draw (0,0) -- ++(1cm,0) -- ++(0,-1.5cm) -- ++(-1cm,0);
            
            % the circle
            \draw (1.75cm,-0.75cm) circle [radius=0.75cm];
            
            % the arrows and labels
            \node[align=center] at (3.13,-0.4) {$r_1$};
            \draw[->] (2.5,-0.75) -- +(1.5cm,0);
            \draw[<-] (0,-0.75) -- +(-1.5cm,0) node[left] {$\{J_t,t\geq0\}$};
            \node[align=center] at (1.1cm,-2cm) {\textit{Upstream server}};
            
            % the rectangle with vertical rules
            \draw (4,0) -- ++(1cm,0) -- ++(0,-1.5cm) -- ++(-1cm,0);
            
            % the circle
            \draw (5.75cm,-0.75cm) circle [radius=0.75cm];
            
            %the arrow
            \draw[->] (6.5,-0.75) -- +(1.5cm,0);
            
            \node[align=center] at (7.13,-0.4) {$r_2$};
            \node[align=center] at (5.2cm,-2cm) {\textit{Downstream server}};
            \end{tikzpicture}
        \caption{\textit{A diagram of the fluid tandem queueing system that we consider in this paper.}}
        \label{fig:diagram}
    \end{center}
\vspace{-21pt}
\end{figure}
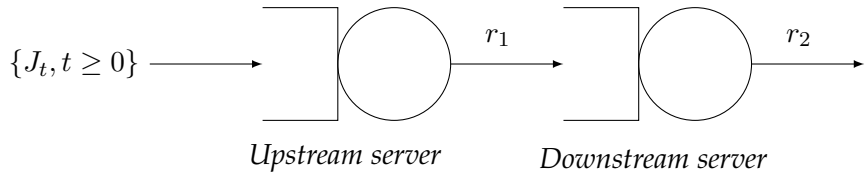

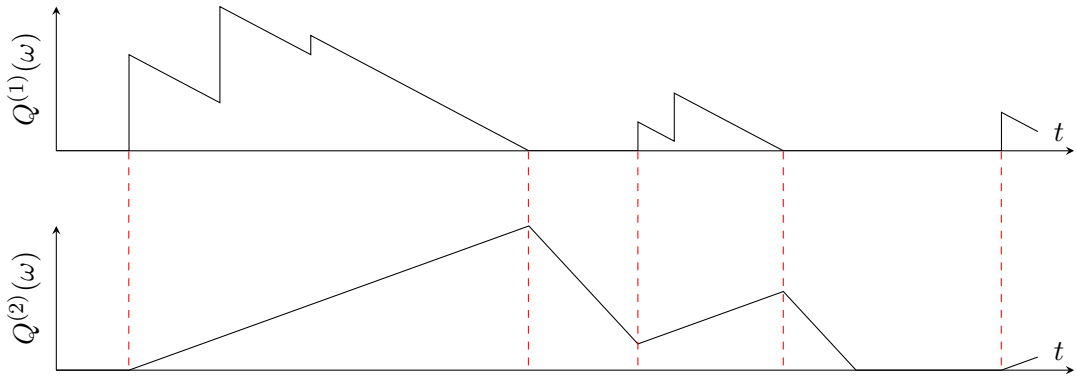
\begin{figure}[htbp]
\centering
    \begin{tikzpicture}
    \begin{groupplot}[group style={group size=1 by 2}, width=\textwidth, height=3.5cm]
        \nextgroupplot[
        xmin=0.8,
        xmax=3.6,
        ymin=0,
        xtick=\empty,
        ytick=\empty,
        ylabel=$Q^{(1)}(\omega)$,
        xlabel=$t$,
        axis lines=middle,
        ylabel absolute,
        ylabel style={yshift=-0.8cm}
        ]
    \addplot[color=black,solid]
     table[row sep=crcr]{%
    0 0\\
    1 0\\
    1 1\\
    1.25 0.5\\
    1.25 1.5\\
    1.5 1\\
    1.5 1.2\\
    2.1 0\\
    2.4 0\\
    2.4 0.3\\
    2.5 0.1\\
    2.5 0.6\\
    2.8 0\\
    3 0\\
    3.4 0\\
    3.4 0.4\\ 
    3.5 0.2\\
    };
    \nextgroupplot[
        xmin=0.8,
        xmax=3.6,
        xtick=\empty,
        ymin=0,
        ytick=\empty,
        ylabel=$Q^{(2)}(\omega)$,
        xlabel=$t$,
        axis lines=middle,
        ylabel absolute,
        ylabel style={yshift=-0.8cm}
    ]
    \addplot[color=black,solid]
     table[row sep=crcr]{%
    0 0\\
    1 0\\
    2.1 1.1\\
    2.4 0.2\\
    2.8 0.6\\
    3 0\\
    3.4 0\\
    3.5 0.1\\
    };
    \end{groupplot}
    
    \draw[red,dashed](0.966,-0.007) -- (0.966,-2.91);
    \draw[red,dashed](6.2785,-0.007) -- (6.2785,-2.91);
    \draw[red,dashed](7.734,-0.007) -- (7.734,-2.91);
    \draw[red,dashed](9.667,-0.007) -- (9.667,-2.91);
    \draw[red,dashed](12.5653,-0.007) -- (12.5653,-2.91);
    \end{tikzpicture}
\caption{\textit{A sample path for an arbitrary $\omega\in\Omega$ in the fluid tandem queue. During busy periods of the upstream queue, the downstream server fills up with a net rate of $r_1-r_2$. During idle periods the workload of the downstream server decreases with rate $r_2$.}}
\label{fig:sample}
\end{figure}

In Regime I, the upstream server will have a fixed load of $\rho_1 = {\E J_1}/{r_1}<1$ as $\e\downarrow0$, whereas the load of the downstream server will tend to one: $\rho_2={\E J_1}/{r_2}\uparrow1$ as $\e\downarrow0$. In Regime II, on the contrary, {\it both} the up- and downstream server will have loads that tend to one: $\rho_1,\rho_2\uparrow1$ as $\e\downarrow0$. To avoid the workload from increasing indefinitely, we scale the workloads so as to obtain a non-degenerate limit. Only the queues for which the load is increasing require an appropriate scaling. The specific way in which the workloads should be scaled depends on the type of input {(more specifically, it matters whether the increments have finite variance or not)}; this will {be pointed out in detail later in the paper.}

We now introduce some additional notation. We define, {for $i=1,2$, $X_t^{(i)}=
J_t - r_i t,$} with {$J_t$ in  the class of spectrally-positive L\'evy processes $\mathcal{S}^+$.}
We denote {by} $\phi$ the {\it Laplace exponent} 
\[\phi(\a) = \log \E[e^{-\a X_1^{(1)}}],\]
and the inverse function of $\phi$ by $\phi^{-1}\equiv\psi$. {To ensure} stability, it is required that the average input {rate} is less than the speed of the slowest working server, i.e., $\E J_1 < r_2 $.

\subsection{Useful results on transforms}
\label{app:LST}
We denote the steady state workload of queue $i\in\{1,2\}$ by $Q^{(i)}$. {The theorems stated below, which uniquey  characterize the distributions of the $Q^{(i)}$, play a crucial role} throughout the paper. The following assertions are Thms.\ 3.2, 12.11, and  12.3, respectively, in the book \cite{Mandjes}, and were developed earlier in e.g.\  \cite{debdierol}. Thm.\  \ref{thm:GPK} gives the Laplace-Stieltjes tranform (LST) for the stationary workload if there is only one server, and can be considered to be %niet considered as!
a generalization of the {well-known} Pollaczek-Khinchine formula. The LST for the joint stationary workload in the fluid tandem system is presented in Thm.\ \ref{thm:joint}, which also provides us with the LST for the downstream queue only (Corollary \ref{thm:marginal}). 

\begin{theorem}[Generalized Pollaczek-Khinchine (PK)]
\label{thm:GPK}
Let $J\in\mathcal{S}^+$. For $ s\geq0$,
\[
\E e^{- s Q^{(1)}} = \f{ s \phi'(0)}{\phi( s)}.
\]
\end{theorem}

\begin{theorem}[Two-dimensional PK for fluid tandem]
\label{thm:joint}
Let $J\in\mathcal{S}^+$. For $ s_1,s_2\geq0$,
\[
\E e^{- s_1 Q^{(1)} - s_2 Q^{(2)}}  = \f{-\E X_1^{(2)} s_2}{s_2-\psi(s_2(r_1-r_2))}\f{\psi(s_2(r_1-r_2))- s_1}{(r_1-r_2)s_2-\phi( s_1)}.
\]
\end{theorem}

\begin{corollary}[One-dimensional PK for fluid tandem]
\label{thm:marginal}
Let $J\in\mathcal{S}^+$. For $ s\geq0$,
\[
\E e^{- s Q^{(2)}} = \f{-\E X_1^{(2)}}{r_1-r_2} \f{\psi( s(r_1-r_2))}{ s-\psi( s(r_1-r_2))}.
\]
\end{corollary}

\begin{remark}
Throughout the remainder of the paper, we assume $J\in\mathcal{S}^+$ and $\E|J_1|<\infty$. It is straightforward to extend our results to spectrally negative input processes $J\in\mathcal{S}^-$, by making use of Laplace-Stieltjes transforms for $\mathcal{S}^-$-processes, which can be found {in e.g.\ Thm.\ 12.12 of \cite{Mandjes}.} 
\end{remark}

\section{Input processes with finite variance}
\label{sec:finvar}
In this section we consider the fluid tandem queue for various types of input processes {that have increments with finite variance}. {Since for Brownian input an explicit analysis can be performed, we  consider this case first (Section \ref{sec:BM}).} Using appropriate expansions, we show in Section \ref{sec:geninput} how these results extend to spectrally-positive L\'evy processes. In both cases we {establish} Regime I and Regime II results. Finally, in Section \ref{sec:numapp} we {provide a numerical comparison} between the Regime I and Regime II approximations. 

\subsection{Brownian input}
\label{sec:BM}
Assume that the input is Brownian, that is, $J_t = \mu t + \sigma W_t$, where $W$ denotes a standard Brownian motion. Then we have
\[
\phi(s) =\log \E[e^{- s J_1}] =\log \E e^{- s(\s W_1 - r)} = \f12\s^2 s^2+ r s,
\]
and after some elementary algebra we find that the inverse is given by 
\begin{equation}
\label{eq:psi}
\psi(s) = -\f{r}{\s^2}+\f1\s\sqrt{\f{r^2}{\s^2}+2s}.
\end{equation}

\subsubsection*{Regime I}
In this case the upstream server has a fixed load $\rho_1<1$ and the downstream server has a load that {tends} to one. Therefore we have to scale the workload of the downstream server: {to obtain a non-degenerate limit, we scale by $\e$}. Relying on Thm.\ \ref{thm:joint}, 
\begin{eqnarray}
\label{eq:2dimlst}
\E e^{- s_1 Q^{(1)} -  s_2 \e Q^{(2)}} &=&\f{\e^2  s_2 }{\e s_2 -\psi(\e s_2 ( r-\e))} \f{\psi(\e s_2 ( r-\e))- s_1}{( r-\e) s_2 \e-\phi( s_1)}\\ \nonumber
&=&\f{1}{( r-\e) s_2 \e- s_1 r-\f12 s_1^2\s^2}\cdot\\
&&\f{(\e s_2 - s_1)\f{1}{\s}\sq{\f{ r^2}{\s^2}+2 s_2 \e( r-\e)}- s_1 s_2 \e-\f{ s_1 r}{\s^2}+\f{ s_2 \e( r-\e)}{\s^2}}{ s_2 +\f{2}{\s^2}},\nonumber
\end{eqnarray}
yielding the following proposition.
\begin{proposition}
Suppose that the input process is a Brownian motion. Then, in Regime~I, the joint stationary workload in heavy traffic is given by
\begin{equation} 
\label{eq:osq}
\E e^{- s_1 Q^{(1)} -  s_2 \e Q^{(2)}} \stackrel{\e\downarrow0}{\longrightarrow}  \f{2 r/\s^2}{2 r/\s^2+ s_1} \f{2/\s^2}{2/\s^2+ s_2 }.
\end{equation}
\end{proposition}

In particular, this implies that the distribution of $\e Q^{(2)}$ converges to an exponential distribution with rate $2/\s^2$, which equals the distribution of the total workload. {Remarkably, $Q^{(1)}$ and $\e Q^{(2)}$ turn out to be asymptotically independent in the limit $\e\downarrow0$.}

Although this asymptotic independence is {an interesting finding} from a theoretical point of view, it has the intrinsic drawback that the original dependency structure is lost. {Another drawback of} this approximation, is that it leads to significant errors if $\rho_1$ is large as well, as will be illustrated in Section \ref{sec:numapp}. This prompts us to consider Regime II. 

\subsubsection*{Regime II}

In this regime we scale both workloads, and we choose the service rates as explained in~Section \ref{sec:model}. Thus, we take $r=\g\e$ in Eqn.\ \eqref{eq:psi}, so as to obtain
\begin{equation}
\label{eq:psi2}
\psi(s) = -\f{\g\e}{\s^2} + \f1\s\sq{\f{\g^2\e^2}{\s^2}+2 s}.
\end{equation}
By Thm.\ \ref{thm:joint},
\[
\E e^{- s_1\e Q^{(1)}-s_2\e Q^{(2)}} = 
\f{-\E X_1^{(2)}s_2\e}{s_2\e-\psi(s_2\e(r_1-r_2))}
\f{\psi(s_2\e(r_1-r_2))- s_1\e}{(r_1-r_2)s_2\e-\phi( s_1\e)}.
\]
Using Eqn.\ \eqref{eq:psi2} yields
\[
\E e^{- s_1 \e Q^{(1)}-s_2 \e Q^{(2)}} =\f{s_2}{s_2+\f{\g}{\s^2}-\f{1}{\s}\sq{\f{\g^2}{\s^2}+2(\g-1)s_2}} \f{-\f{\g}{\s^2}+\f{1}{\s}\sq{\f{\g^2}{\s^2}+2(\g-1)s_2}- s_1}{(\g-1)s_2-\g s_1-\f12 s_1^2\s^2},
\]
where it should be noted that the expression on the right-hand side {does not contain any $\e$ anymore}. This indicates that, for Brownian input, the joint distribution in the heavy-traffic limit is of the same type as the distribution for {`non-heavy-traffic loads'} loads $\rho_1$ and $\rho_2$. After further simplification we obtain
\begin{equation} 
\label{eq:joint}
\E e^{- s_1\e Q^{(1)}- s_2\e Q^{(2)}} =
\f{s_2(\g-2- s_1\s^2)- s_1\g+(s_2- s_1)\s\sqrt{\f{\g^2}{\s^2}+2(\g-1)s_2}}{(s_2\s^2+2)((\g-1)s_2-\g s_1-\frac{1}{2} s_1^2\s^2)}.
\end{equation}
We can find the marginal distributions of the stationary workload of the first and second queue by plugging in $s_2=0$, resp.\ $s_1=0$. This yields 
\begin{equation}
\label{eq:LSTQ1}
\E e^{- s\e Q^{(1)}}= \f{2\g/\s^2}{2\g/\s^2+ s},
\end{equation}
\begin{equation}
\label{eq:LSTQ2}
\E e^{- s \e Q^{(2)}} = \f{1}{\g-1}   \f{-2+\g+\sq{\g^2+2 s\s^2(\g-1)}}{2+ s\s^2}.
\end{equation}

\iffalse
If we substitute $s=s_1=s_2$ in Equation \eqref{eq:joint}, then we find
\[
\E e^{- s (Q^{(1)}+Q^{(2)})} = \f{2/\sigma^2}{2/\sigma^2 +  s},
\]
that is, $Q^{(1)}+Q^{(2)}\eqd\textrm{Exp}(\f{2}{\s^2})$. In other words, the whole system behaves as the first queue, which has the output rate of the second queue. \fi

After lengthy but elementary calculations we obtain
\[
\E[ Q^{(1)} Q^{(2)}] = \lim_{ s_1\downarrow0}\lim_{s_2\downarrow0} \frac{\partial^2}{\partial s_1\partial s_2} \E e^{- s_1 Q^{(1)}-s_2 Q^{(2)}} = \f{\g^2-1}{4\g^3} \s^4.
\]
{Using that $\E[Q^{(1)}]= \frac{1}{2}{\s^2}/{\g}$ and $\E[Q^{(2)}]= \f{1}{2}\s^2{(\g-1)}/{\g}$,}
\begin{equation}
\label{eq:corr}
\textrm{Cov}(Q^{(1)}, Q^{(2)}) =\s^4\Big(\f{\g^2-1}{4\g^3} - \f{\g-1}{4\g^2}\Big) = \f{\g-1}{4\g^3}\s^4.
\end{equation}
To calculate the correlation coefficient, we also {compute the variances}. Since $Q^{(1)}$ has an $\textrm{Exp}(2\g/\s^2)$ distribution, its variance is given by $\text{Var}(Q^{(1)})=\f{1}{4}{\s^4}/{\g^2}$. By making  use of the LST of $Q^{(2)}$, we also find 
\[
\text{Var}(Q^{(2)}) = \f{(\g-1)^2(\g+2)\s^4}{4\g^3}.
\] 
It now follows that the correlation coefficient is given by
\begin{equation}
\label{eq:cg}
\textrm{Corr}(Q^{(1)},Q^{(2)}) =c(\g)= \f{1}{\sqrt{\g(\g+2)}}.
\end{equation}

Observe that, when decreasing $\g$ from $\infty$ to $1$, $c(\g)$ increases from $0$ to ${1}/{\sqrt{3}}$. This result is in line with Corollary 4.1 in \cite{Kella1993}: there $c(\g)$ is studied without heavy traffic, and it is concluded that $c(\gamma)\in(0,1/\sqrt{3})$.
In the introduction we already argued why $c(\g)$ is anticipated to be positive, but it can also be seen that $c(\g)$ decreases in  $\g$. {Indeed, as $\g$ grows, the service rate in the upstream server increases. This implies that it becomes more likely that the downstream server has a large workload, while the workload in the first server may be relatively small due to its fast service.}

\subsection{General input}
\label{sec:geninput}
We now extend the results for the Brownian case in the previous section to spectrally-positive L\'evy input. Again we consider both regimes, starting with Regime I.

\label{sec:MG1}
\subsection*{Regime I}
In this section we prove the following main result.

\begin{proposition}
\label{prop:prop1}
Let the input process $J\in\mathcal{S}^+$ be such that $\Var J_1 = \s^2 <\infty$. Then, in Regime I, the stationary workloads of the up- and downstream queue are asymptotically independent, {with $Q^{(1)}$ given by Thm.\ \ref{thm:GPK}, and $Q^{(2)} \stackrel{\rm d}{=} {\rm Exp}(\f{2}{\sigma^2})$.}
\end{proposition}

{To prove} this proposition, we require the following lemma.

\begin{lemma}
\label{lem:psigen}
Let 
\begin{equation}
\label{eq_phi}
\phi(s) = sr + \f12 \s^2 s^2 +  K_1 s^{\eta_1} + o(s^{\eta_1}),\end{equation}
with $\eta_1>2$. Then the inverse function $\psi$ with argument $s \e(r-\e)$, {satisfies}, for $\e\downarrow0$,
\[
\psi(s\e(r-\e)) = s \e - \f1r s \e^2 - \f{\s^2}{2r} s^2 \e^2 + o(\e^2).
\]
\end{lemma}
\begin{proof}[Proof of Lemma \ref{lem:psigen}]
Suppose that 
\[
\psi(s\e(r-\e)) = C_1 s \e + C_2 s\e^2 + C_3 s^2 \e^2 + o(\e^2).
\]
Consider
\begin{eqnarray*}
\phi(\psi(s\e(r-\e))) - s\e(r-\e) &= &\psi(s\e(r-\e))r + \f12 \s^2\psi(s\e(r-\e))^2 \\
&&+\, K_1\psi(s\e(r-\e))^{\eta_1} - s\e(r-\e) + o(\e^2)\\
&=&(C_1 r -r)s\e + (r C_2 +1)s\e^2  + \e^2(\s^2 C_1 + 2r C_3)\f12s^2  + o(\e^2).
\end{eqnarray*}
For $\psi$ to be the inverse of $\phi$ for $\e \downarrow0$, we {equate} the above to zero. This is achieved by taking the constants $C_1=1$, $C_2 = -\f1r$ and $C_3 = -\f{\s^2}{2r}$. This proves the lemma. 
\end{proof}

At first glance, it may be unclear why $\psi$ in Lemma \ref{lem:psigen} has this specific form. However, in case of e.g.\ compound Poisson input, this shape arises naturally, as is demonstrated in Example \ref{ex:MG1} below. We first prove the main result.

\begin{proof}[Proof of Proposition \ref{prop:prop1}]
Assume that $\Var J_1 = \sigma^2<\infty$. We first develop a general expansion for $\phi$. From the definition of $\phi$, we have $\phi(s) = s r_1 + \log \E e^{-sJ_1}.$
{Note that $\phi(s)$ is linear in $r$ in $s=0$:}
\[
\phi'(0) = r_1 - \E J_1 = \E J_1 + r - \E J_1 = r.
\]
Now note that $\phi''(0) = \Var J_1 = \s^2$. This means that the coefficient of $s^2$ must be $\f12\s^2$. Upon combining all of the above, we see that necessarily
\[
\phi(s) = sr +\f12\s^2 s^2 +   o(s^2).
\]
We can write
\[
\phi(s) = sr +\f12 \s^2 s^2 +   K_1 s^{\eta_1} + o(s^{\eta_1}),
\]
for some $K_1\in\mathbb{R}$, where $\eta_1=3$ corresponds to the existence of a finite third moment, and $2<\eta_1<3$ corresponds to an infinite third moment. 
{It thus follows that $\phi$ as in (\ref{eq_phi}) covers all input processes with finite second moment. }Therefore we can use the functions $\phi$ and $\psi$ in Lemma \ref{lem:psigen}, and apply them to Thm.\ \ref{thm:joint}. By scaling only the workload of the second queue by a factor $\e$ and taking the heavy-traffic limit, we find
\[
\lim_{\e\downarrow0} \E e^{- s_1 Q^{(1)} - s_2\e Q^{(2)}} = \f{s_1 r}{\phi(s_1)} \f{1}{1+\f12\s^2 s_2} = \f{s_1 \phi'(0)}{\phi(s_1)} \f{1}{1+\f12\s^2 s_2}  .
\]
The result follows.
\end{proof}

\begin{example}
\label{ex:MG1}
Suppose that the input process is a compound Poisson process in which the first two moments of the job sizes are finite: $\E B, \E B^2 <\infty$. The goal is to find an asymptotic expression for $\psi(s\e(r-\e))$ as $\e\downarrow0$, while $s\geq0$ is fixed. The proof of Lemma~\ref{lem:psigen} is by validation. How such an expression for $\psi(s\e(r-\e))$ can be constructed becomes clear in this example. We approach the problem in the following steps:

\begin{itemize}
\item Derive the Tak\'acs equation {(describing the LST $\pi$ of the busy period in an M/G/1 queue)} with service rate equal to $r_1$;
\item Use this Tak\'acs equation to express $\psi$ in terms of $\pi$;
\item Expand $\pi$, which yields an expansion for $\psi$.
\end{itemize}

Since we have a compound Poisson input process, the Laplace exponent is given by
\begin{equation}
\label{eq:laplace exponent}
\phi( s) =  s r_1 - \l + \l b( s),
\end{equation}
with $b(s) = \E e^{-s B}$. Let $\tau^0$ denote the busy period started by a job arriving at an empty system. {Using the standard argumentation, it turns out that}
\begin{equation} 
\label{eq:pib}
\pi(s) = b\Big(\f{1}{r_1} (\l-\l\pi(s)+s)\Big);
\end{equation} 
this functional equation is well-known for $r_1=1$,  cf.\ Section 1.3 in \cite{Takacs}, but it can be readily extended to general $r_1$. Eqns.\ \eqref{eq:laplace exponent} and \eqref{eq:pib} imply
\[
\f{1}{\l} \phi\Big(\f{1}{r_1}(\l-\l\pi( s)+ s)\Big) - \f{ s}{\l} =b\Big(\f{1}{r_1}(\l-\l\pi( s)+ s)\Big) - \pi( s) = 0.
\]
Applying the inverse function $\psi$, we obtain
\begin{equation}
\label{eq:psipi}
\psi( s) = \f{\l-\l\pi( s)+ s}{1+r}.
\end{equation}

Now we will find an expansion for $\pi$, which {in turn} yields an expansion for $\psi$. Using Eqn.\ \eqref{eq:pib} and some elementary calculus, we find
\begin{align*}
\pi'(0) = \f{-\E B}{r_1 - \l \E B} \quad\text{and}\quad \pi''(0) & = \f{r_1^2 \E B^2}{(r_1-\l\E B)^3}.
\end{align*}
Recall that the above quantities are finite {by the conditions we imposed on the moments of $B$}, and since the loads are assumed to be less than one. Therefore,
\[
\pi(s) = 1 - \f{\E B}{r_1 - \l \E B} s + \f12 \f{r_1^2 \E B^2}{(r_1-\l\E B)^3} s^2 + o(s^2).
\]
Substituting this into Eqn.\ \eqref{eq:psipi} yields
\[
\psi(s) = \f1r s - \f12 \f{\l \E B^2}{r^2} s^2 + o(s^2).
\]
It follows that
\[
\psi(s\e(r_1-r_2)) = \psi(s\e(r-\e)) = s\e - \f1r\Big(s\e^2 + \f12 \l \E B^2 s^2 \e^2\Big)+ o(\e^2).
\]
Noting that $\l\E B^2=\s^2$, we find the structure of $\psi(s\e(r_1-r_2))$ as in Lemma \ref{lem:psigen}.
\end{example}

\subsection*{Regime II}
In the following we consider the corresponding Regime II result. It should be noted that the methodology is similar to the one for Regime I. However, since the $\e$ now plays a different role, we cannot use Lemma \ref{lem:psigen}, but we develop Lemma \ref{lem:psiK} instead.

\begin{proposition}
\label{prop:MG1R2}
{Let the input process $J\in\mathcal{S}^+$ be such that $\Var J_1 =\s^2<\infty$}. Then, in Regime II, the joint scaled workload is given by
\[
\lim_{\e\downarrow0} \E e^{- s_1\e Q^{(1)}- s_2\e Q^{(2)}} =
\f{s_2(\g-2- s_1\s^2)- s_1\g+(s_2- s_1)\s\sqrt{\f{\g^2}{\s^2}+2(\g-1)s_2}}{(s_2\s^2+2)((\g-1)s_2-\g s_1-\frac{1}{2} s_1^2\s^2)}.\]
\end{proposition}

\begin{remark}
Note that the result in Prop.\ \ref{prop:MG1R2} corresponds to Eqn.\ \eqref{eq:joint}, i.e., the LST we found in case of Brownian input, except now we do take a proper heavy-traffic limit, whereas Eqn.\ \eqref{eq:joint} holds for all $\e>0$. 
\end{remark}

\iffalse
\begin{remark}
The naive approach to prove Proposition \ref{prop:MG1R2}, would be to use the approximation of Regime I for $\psi(s\e(r-\e))$ as in Lemma \ref{lem:psigen}, in which one could substitute $r=\g\e$. However, the resulting approximation does not hold uniformly over $s$ for $\e\downarrow0$. Therefore we really have to find a different expression for $\psi(s\e(r_1-r_2))$, as shown in Lemma \ref{lem:psiK}. This shows that Regime I results do not translate directly into Regime II results and require a different approach.

In Lemma \ref{lem:psigen} (Regime I), the shape of $\psi(s\e(r_1-r_2))$ was initially found by the authors by considering compound Poisson input as in Example \ref{ex:MG1}. The form of $\psi(s\e(r_1-r_2))$ in Lemma \ref{lem:psiK} below (Regime II), is in line with the form of the exact inverse for Brownian input, as in Section \ref{sec:BM}.
\end{remark}\fi

\begin{lemma}
\label{lem:psiK}
Let
\[
\phi(s) = s\e + \f12 s^2 +  K_1 s^{\eta_1} + o(s^{\eta_1}),
\]
for some constant $K_1\in\mathbb{R}$ and $2<\eta_1\leq 3$. Then, asymptotically for $\e\downarrow0$, we have
\[
\psi( s\e^2(\g-1)) = - \e + \e\sq{1+2 s(\g-1)} + o(\e).
\]
\end{lemma}
\begin{proof}[Proof of Lemma \ref{lem:psiK}]
Suppose that 
\[
\psi( s\e^2(\g-1)) = - \e + \e\sq{1+2 s(\g-1)} + K_2( s) \e^{2\eta_2},
\]
for some function $K_2$ of $ s$ (and independent of $\e$) and for some constant $\eta_2$. If we show that \[\lim_{\e\downarrow0}\f{\phi(\psi(s\e^2(\g-1)))}{s\e^2(\g-1)}=1\] and $\eta_2\geq\f12$, then we have proved the lemma.

Indeed, for all $ s\geq0$,
\begin{eqnarray*}
\phi(\psi( s\e^2(\g-1))) - s\e^2(\g-1) &=& - \e^2 + \e^2\sq{1+2 s(\g-1)}+K_2( s)\e^{2\eta_2+1}+o(\e^{2\eta_1})\\
&&+\, \f12 \Big(-\e+\e\sq{1+2 s(\g-1)} + K_2( s) \e^{2\eta_2} \Big)^2\\
&&+\,K_1\Big(-\e+\e\sq{1+2 s(\g-1)}+K_2( s)\e^{2\eta_2}\Big)^{\eta_1}\\
&&-\, s\e^2(\g-1).
\end{eqnarray*}
Hence, after simplification, we see that the following should hold for all $ s\geq0$:
\begin{eqnarray}
\nonumber
\lefteqn{\f12 K_2( s)^2 \e^{4\eta_2} + K_2( s)\e^{2\eta_2+1}\sq{1+2 s(\g-1)}}\\
&&+\,K_1\Big(-\e+\e\sq{1+2 s(\g-1)}+K_2( s)\e^{2\eta_2}\Big)^{\eta_1}+o(\e^{2\eta_1}) =0.\label{eq:lemeq}
\end{eqnarray}

\textbf{Case 1}.
If $\eta_2<\f12$, then using \eqref{eq:lemeq} we obtain for all $ s\geq0$,
\[
\f12 K_2( s)^2 \e^{4\eta_2} +K_1 K_2( s)^{\eta_1} \e^{2\eta_1\eta_2}+o(\e^{2\eta_1}) =0,
\]
which holds for $\epsilon\downarrow0$ if and only if $4\eta_2 = 2\eta_1\eta_2$. This implies $\eta_1=2$, but this contradicts $\eta_1>2$. We conclude that $\eta_2\geq\f12$. 
 
\textbf{Case 2}.
If $\eta_2=\f12$, then we can write
\[
\f12 K_2( s)^2 \e^{2} + K_2( s)\e^{2}\sq{1+2 s(\g-1)} + o(\e^2)=0,
\]
which is solved by $K_2=0$. Note that the conclusion of the lemma holds in this case.

\textbf{Case 3}.
If $\eta_2>\f12$, then we can write for all $ s\geq0$,
\[
K_2( s)\e^{2\eta_2+1}\sq{1+2 s(\g-1)}+K_1\e^{\eta_1}\Big(-1+\sq{1+2 s(\g-1)}\Big)^{\eta_1} +  o(\e^{\min\{4\eta_2,\eta_1\}}) =0.
\]
We see that we have to make sure that $2\eta_2 +1 = \eta_1$, since the equation has to hold for all $ s\geq0$. So define $\eta_2=\f 12 ({\eta_1-1})$. Then, for all $ s\geq0$,
\[
K_2( s)\sq{1+2 s(\g-1)}+K_1\Big(-1+\sq{1+2 s(\g-1)}\Big)^{\eta_1} +  o(1) =0.
\]
The conclusion of the lemma holds in this case,   noting that $\eta_2>\f12$, where 
\[
K_2( s) := -\f{K_1\Big(-1+\sq{1+2 s(\g-1)}\Big)^{\eta_1}}{\sq{1+2 s(\g-1)}}.
\] This proves the claim.
\end{proof}

\begin{proof}[Proof of Proposition \ref{prop:MG1R2}]
This result follows from Lemma \ref{lem:psiK} and Thm.\ \ref{thm:joint}, and taking the limit $\e\downarrow0$. The calculations are similar to those in the Brownian case, except there are some additional terms of small order $\e$ that cancel in the heavy-traffic limit.
\end{proof}

\subsection{Numerical approximations for exponential jobs}
\label{sec:numapp}
\begin{example}[Comparison of Regime I and Regime II]
Suppose that we have a system with compound Poisson input with exponential jobs, with $\l=1$, $\mu=1$. {By Eqn.\ \eqref{eq:osq}}, one obtains the Regime I approximation $Q^{(2)} \eqd \textrm{Exp}(\e)$. Due to Equation \eqref{eq:LSTQ2}, the Regime II approximation entails numerically inverting
\[
\E e^{- s Q^{(2)}} = \f{1}{\g-1}   \f{-2+\g+\sq{\g^2+2\f s\e\s^2(\g-1)}}{2+\f s\e\s^2}.
\]

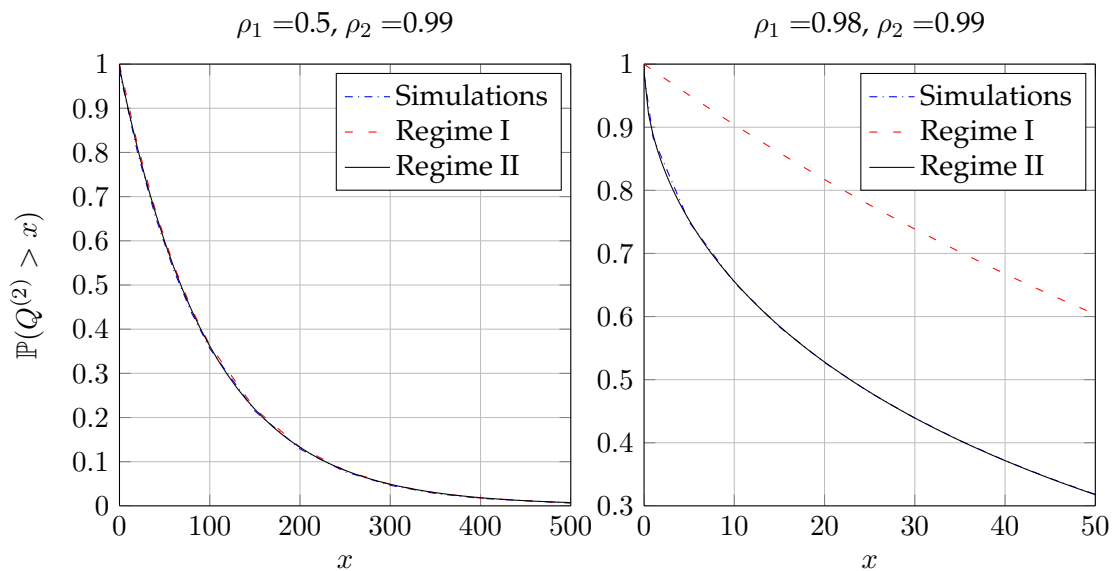
\begin{figure}[b!]
\centering
\begin{subfigure}{.5\textwidth}
  \centering
\newlength
\figureheight 
\newlength
\figurewidth 
\setlength
\figureheight{6.3cm} 
\setlength
\figurewidth{6cm} 
\begin{tikzpicture}

\begin{axis}[%
width=\figurewidth,
height=0.933\figureheight,
at={(0\figurewidth,0\figureheight)},
scale only axis,
separate axis lines,
every outer x axis line/.append style={black},
every x tick label/.append style={font=\color{black}},
xmin=0,
xmax=500,
xtick={  0,   100,  200,  300,  400,  500},
xlabel={$x$},
xmajorgrids,
every outer y axis line/.append style={black},
every y tick label/.append style={font=\color{black}},
ymin=0,
ymax=1,
ytick={  0, 0.1, 0.2, 0.3, 0.4, 0.5, 0.6, 0.7, 0.8, 0.9,   1},
ylabel={$\Pb(Q^{(2)}>x)$},
ymajorgrids,
axis background/.style={fill=white},
title={$\rho_1=$0.5, $\rho_2=$0.99},
legend style={legend cell align=left,align=left,draw=black}
]

\addplot[color=blue,dashdotted]
 table[row sep=crcr]{%
0.01	0.98966235\\
1	0.974746578\\
20	0.799958764\\
40	0.653489076\\
60	0.533933708\\
80	0.436203806\\
100	0.356405396\\
150	0.214737326\\
200	0.128810748\\
250	0.077289086\\
300	0.046514592\\
350	0.027973118\\
400	0.01696124\\
500	0.006282238\\
};
\addlegendentry{Simulations};

\addplot [color=red,loosely dashed]
  table[row sep=crcr]{%
0.01	0.999898995000338\\
1	0.989949833766045\\
20	0.81707842114073\\
40	0.667617146293829\\
60	0.545495563820241\\
80	0.445712654025515\\
100	0.364182191633613\\
150	0.219774883100251\\
200	0.132628668703061\\
250	0.0800380985934706\\
300	0.0483009992417292\\
350	0.0291484501599573\\
400	0.0175903637619464\\
500	0.00640609722645812\\
};
\addlegendentry{Regime I};

\addplot [color=black,solid]
  table[row sep=crcr]{%
0.01	1\\
0.1	0.997024287\\
1.08	0.98293592\\
2.07	0.9719698\\
3.06	0.961736437\\
4.05	0.951860574\\
5.04	0.942211138\\
6.03	0.9327289\\
7.02	0.923383132\\
8.01	0.914156372\\
9	0.905038001\\
9.99	0.896021182\\
10.9	0.887818563\\
12.8	0.87094549\\
14.7	0.854402365\\
16.6	0.838178321\\
18.5	0.822265049\\
20.4	0.806655517\\
22.3	0.79134337\\
24.2	0.776322649\\
28	0.747132861\\
29.9	0.7329529\\
30	0.7322141\\
40	0.66197791\\
50	0.598489825\\
60	0.541101432\\
70	0.489226692\\
80	0.442335875\\
90	0.399950145\\
100	0.36163667\\
110	0.327004204\\
120	0.295699091\\
130	0.267401651\\
140	0.241822918\\
150	0.21870169\\
160	0.19780186\\
170	0.178910003\\
180	0.161833203\\
190	0.146397075\\
200	0.132443989\\
210	0.119831462\\
220	0.108430697\\
230	0.098125273\\
240	0.088809955\\
250	0.080389615\\
260	0.072778269\\
270	0.065898191\\
280	0.059679125\\
290	0.054057562\\
300	0.048976097\\
310	0.044382839\\
320	0.040230883\\
330	0.036477831\\
340	0.033085358\\
350	0.030018821\\
360	0.027246904\\
370	0.024741302\\
380	0.022476429\\
390	0.020429156\\
400	0.018578577\\
410	0.016905794\\
420	0.015393726\\
430	0.014026931\\
440	0.012791453\\
450	0.011674674\\
460	0.010665191\\
470	0.009752696\\
480	0.008927869\\
490	0.008182289\\
500	0.007508341\\
};
\addlegendentry{Regime II};

\end{axis}
\end{tikzpicture}%
\end{subfigure}%
\begin{subfigure}{.5\textwidth}
  \centering
\setlength
\figureheight{6.3cm} 
\setlength
\figurewidth{6cm}
\begin{tikzpicture}

\begin{axis}[%
width=\figurewidth,
height=0.933\figureheight,
at={(0\figurewidth,0\figureheight)},
scale only axis,
separate axis lines,
every outer x axis line/.append style={black},
every x tick label/.append style={font=\color{black}},
xmin=0,
xmax=50,
xtick={ 0, 10, 20, 30, 40, 50},
xlabel={$x$},
xmajorgrids,
every outer y axis line/.append style={black},
every y tick label/.append style={font=\color{black}},
ymin=0.3,
ymax=1,
ytick={0, 0.1, 0.2, 0.3, 0.4, 0.5, 0.6, 0.7, 0.8, 0.9,1},
ylabel={},
ymajorgrids,
axis background/.style={fill=white},
title={$\rho_1=$0.98, $\rho_2=$0.99},
legend style={legend cell align=left,align=left,draw=black}
]
\addplot [color=blue,dashdotted]
  table[row sep=crcr]{%
0.01	0.983132754\\
1	0.883927872\\
5	0.749853232\\
10	0.653803106\\
15	0.583540756\\
20	0.527062844\\
25	0.479783892\\
30	0.439076044\\
35	0.403449346\\
40	0.371788164\\
45	0.343549844\\
50	0.318138456\\
};
\addlegendentry{Simulations};

\addplot [color=red,loosely dashed]
  table[row sep=crcr]{%
0.01	0.999898995000338\\
1	0.989949833766045\\
5	0.950749126896934\\
10	0.903923902295282\\
15	0.859404860888509\\
20	0.817078421140731\\
25	0.776836595505876\\
30	0.738576714918798\\
35	0.702201166855453\\
40	0.66761714629383\\
45	0.634736418940282\\
50	0.60347509611716\\
};
\addlegendentry{Regime I};

\addplot [color=black,solid]
  table[row sep=crcr]{%
0.0001	1\\
0.0131	0.987160418\\
0.5001	0.9206998\\
1	0.888018852\\
1.5	0.863051933\\
2	0.842102641\\
2.5	0.823733305\\
3	0.807205247\\
3.5	0.792078853\\
4	0.778067167\\
4.5	0.764970499\\
5	0.752643184\\
5.5	0.740975052\\
6	0.729880350\\
6.5	0.719290754\\
7	0.709150759\\
7.5	0.699414534\\
8	0.690043704\\
8.5	0.681005741\\
9	0.672272788\\
9.5	0.663820759\\
10	0.655628661\\
10.5	0.647678055\\
11	0.639952646\\
11.5	0.632437937\\
12	0.625120969\\
12.5	0.617990095\\
13	0.611034799\\
13.5	0.604245551\\
14	0.597613674\\
14.5	0.591131246\\
15	0.584791006\\
15.5	0.578586280\\
16	0.572510916\\
16.5	0.566559225\\
17	0.560725939\\
17.5	0.555006160\\
18	0.549395330\\
18.5	0.543889197\\
19	0.538483785\\
19.5	0.533175370\\
20	0.527960460\\
21	0.517798192\\
22.1	0.507008119\\
23.2	0.496597521\\
24.3	0.486541225\\
25.4	0.476816772\\
26.5	0.467404022\\
27.6	0.458284816\\
28.7	0.449442716\\
29.8	0.440862774\\
30.9	0.432531350\\
32	0.424435951\\
33.1	0.416565094\\
34.2	0.408908200\\
35.3	0.401455487\\
36.4	0.394197892\\
37.5	0.387126993\\
38.6	0.380234948\\
39.7	0.373514439\\
40.8	0.366958623\\
41.9	0.360561086\\
43	0.354315808\\
44.1	0.348217131\\
45.2	0.342259722\\
46.3	0.336438553\\
47.4	0.330748872\\
48.5	0.325186185\\
49.6	0.319746235\\
50	0.317797706\\
 };
\addlegendentry{Regime II};
\end{axis}
\end{tikzpicture}%
\end{subfigure}
\caption{\textit{Varying $\rho_1$ while keeping $\rho_2=0.99$. It appears that the Regime II approximation is almost perfect, and the Regime I approximation becomes worse the higher $\rho_1$ becomes.}}
\label{fig:HT1}
\end{figure}

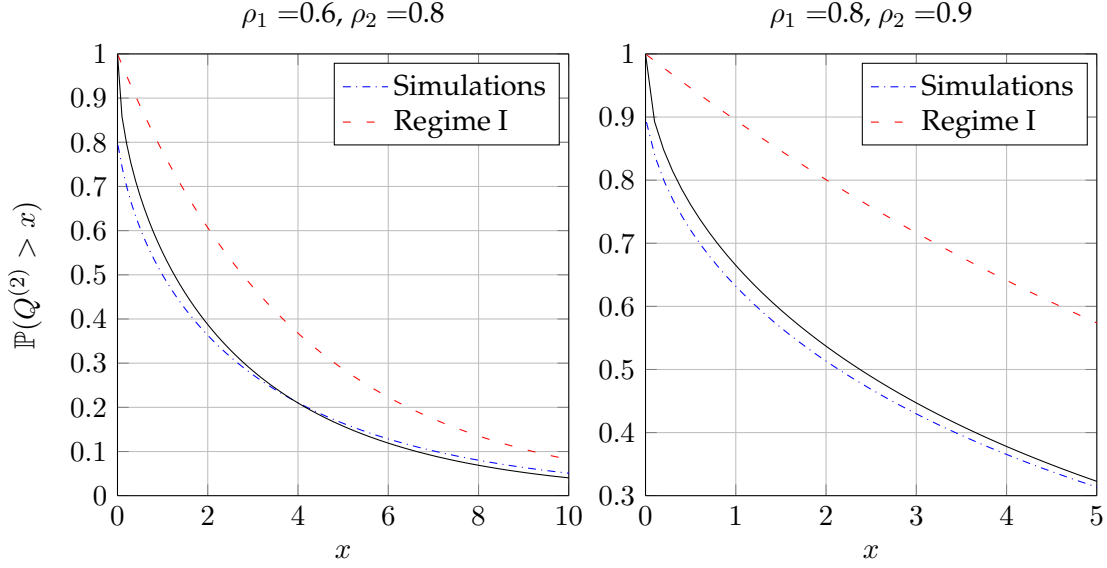
\begin{figure}[ht!]
\centering
\begin{subfigure}{.5\textwidth}
\centering
\setlength
\figureheight{6.3cm} 
\setlength
\figurewidth{6cm}
\begin{tikzpicture}

\begin{axis}[%
width=\figurewidth,
height=0.933\figureheight,
at={(0\figurewidth,0\figureheight)},
scale only axis,
separate axis lines,
every outer x axis line/.append style={black},
every x tick label/.append style={font=\color{black}},
xmin=0,
xmax=10,
xtick={0,2,4,6,8,10},
xlabel={$x$},
xmajorgrids,
every outer y axis line/.append style={black},
every y tick label/.append style={font=\color{black}},
ymin=0,
ymax=1,
ytick={  0, 0.1, 0.2, 0.3, 0.4, 0.5, 0.6, 0.7, 0.8, 0.9,   1},
ylabel={$\Pb(Q^{(2)}>x)$},
ymajorgrids,
axis background/.style={fill=white},
title={$\rho_1=$0.6, $\rho_2=$0.8},
legend style={legend cell align=left,align=left,draw=black}
]
\addplot [color=blue,dashdotted]
  table[row sep=crcr]{%
0.01	0.79358274\\
0.01	0.79358274\\
0.11	0.73955416\\
0.21	0.6968972\\
0.31	0.66123844\\
0.41	0.6304112\\
0.51	0.60305312\\
0.61	0.5783692\\
0.71	0.55579726\\
0.81	0.53501288\\
0.91	0.51566524\\
1.01	0.49760094\\
1.11	0.48064442\\
1.21	0.46467802\\
1.31	0.44956904\\
1.41	0.4352097\\
1.51	0.4215822\\
1.61	0.40858846\\
1.71	0.39618358\\
1.81	0.38432274\\
1.91	0.37295896\\
2.01	0.3620521\\
2.11	0.3515623\\
2.21	0.34149086\\
2.31	0.3318152\\
2.41	0.32249398\\
2.51	0.31351426\\
2.61	0.30487248\\
2.71	0.29650418\\
2.81	0.28844076\\
2.91	0.28066792\\
3.01	0.27314922\\
3.11	0.26588328\\
3.21	0.25883918\\
3.31	0.2520219\\
3.41	0.24542256\\
3.51	0.23906614\\
3.61	0.23289274\\
3.71	0.22689962\\
3.81	0.22109\\
3.91	0.21546766\\
4.01	0.21000918\\
4.11	0.2046905\\
4.21	0.19956926\\
4.31	0.19460588\\
4.41	0.189756\\
4.51	0.18504774\\
4.61	0.18045964\\
4.71	0.17601166\\
4.81	0.17169138\\
4.91	0.16750136\\
5.01	0.16342172\\
5.11	0.15945172\\
5.21	0.15559052\\
5.31	0.15183182\\
5.41	0.14817506\\
5.51	0.14461776\\
5.61	0.1411505\\
5.71	0.13777162\\
5.81	0.13450138\\
5.91	0.13131174\\
6.01	0.12821234\\
6.11	0.125175\\
6.21	0.12224078\\
6.31	0.11937014\\
6.41	0.1165729\\
6.51	0.11383604\\
6.61	0.11116412\\
6.71	0.10856662\\
6.81	0.10603724\\
6.91	0.10356752\\
7.01	0.10117044\\
7.11	0.09882112\\
7.21	0.09653526\\
7.31	0.09430488\\
7.41	0.09212856\\
7.51	0.08999374\\
7.61	0.08792308\\
7.71	0.08590274\\
7.81	0.08393046\\
7.91	0.08201208\\
8.01	0.08014622\\
8.11	0.07831824\\
8.21	0.0765389\\
8.31	0.07480342\\
8.41	0.07309542\\
8.51	0.07142994\\
8.61	0.06981378\\
8.71	0.06822078\\
8.81	0.06667224\\
8.91	0.0651659\\
9.01	0.06369172\\
9.11	0.06224568\\
9.21	0.0608488\\
9.31	0.0594864\\
9.41	0.05815906\\
9.51	0.0568624\\
9.61	0.05558526\\
9.71	0.05434068\\
9.81	0.05313518\\
9.91	0.05194672\\
10.01	0.05079198\\
};
\addlegendentry{Simulations};

\addplot [color=red,loosely dashed]
  table[row sep=crcr]{%
0.01	0.99750312239746\\
0.01	0.99750312239746\\
0.11	0.972874682553454\\
0.21	0.948854321055801\\
0.31	0.925427024396637\\
0.41	0.902578149752926\\
0.51	0.880293415834221\\
0.61	0.858558893956395\\
0.71	0.837360999335754\\
0.81	0.816686482598111\\
0.91	0.796522421497492\\
1.01	0.776856212839313\\
1.11	0.757675564602974\\
1.21	0.738968488258944\\
1.31	0.720723291275541\\
1.41	0.702928569810718\\
1.51	0.685573201584293\\
1.61	0.668646338926159\\
1.71	0.652137401996139\\
1.81	0.63603607217124\\
1.91	0.620332285596178\\
2.01	0.605016226893143\\
2.11	0.590078323026865\\
2.21	0.575509237321158\\
2.31	0.561299863623192\\
2.41	0.54744132061185\\
2.51	0.533924946246618\\
2.61	0.520742292353521\\
2.71	0.507885119344745\\
2.81	0.495345391068622\\
2.91	0.483115269786778\\
3.01	0.471187111275287\\
3.11	0.459553460046784\\
3.21	0.448207044690545\\
3.31	0.437140773327614\\
3.41	0.426347729178153\\
3.51	0.415821166238224\\
3.61	0.405554505063321\\
3.71	0.395541328656001\\
3.81	0.385775378455054\\
3.91	0.376250550423696\\
4.01	0.366960891234347\\
4.11	0.357900594547609\\
4.21	0.349063997383117\\
4.31	0.340445576579986\\
4.41	0.332039945344661\\
4.51	0.323841849883993\\
4.61	0.31584616612145\\
4.71	0.308047896494398\\
4.81	0.300442166830464\\
4.91	0.293024223301021\\
5.01	0.28578942944989\\
5.11	0.278733263295399\\
5.21	0.271851314504006\\
5.31	0.265139281633689\\
5.41	0.258592969445408\\
5.51	0.252208286280946\\
5.61	0.245981241505486\\
5.71	0.239907943013336\\
5.81	0.233984594795235\\
5.91	0.228207494565726\\
6.01	0.222573031449104\\
6.11	0.217077683722505\\
6.21	0.211718016614711\\
6.31	0.206490680159306\\
6.41	0.201392407100844\\
6.51	0.196420010852698\\
6.61	0.191570383505349\\
6.71	0.186840493883836\\
6.81	0.182227385653174\\
6.91	0.17772817547055\\
7.01	0.173340051183139\\
7.11	0.169060270070414\\
7.21	0.164886157129861\\
7.31	0.160815103405015\\
7.41	0.156844564354772\\
7.51	0.152972058262975\\
7.61	0.149195164687255\\
7.71	0.14551152294618\\
7.81	0.141918830643747\\
7.91	0.138414842230317\\
8.01	0.134997367599066\\
8.11	0.131664270717101\\
8.21	0.128413468290371\\
8.31	0.125242928461535\\
8.41	0.12215066953999\\
8.51	0.11913475876325\\
8.61	0.116193311088902\\
8.71	0.113324488016397\\
8.81	0.110526496437928\\
8.91	0.107797587517676\\
9.01	0.105136055598731\\
9.11	0.102540237137004\\
9.21	0.100008509661456\\
9.31	0.097539290759999\\
9.41	0.0951310370904406\\
9.51	0.0927822434158417\\
9.61	0.0904914416636959\\
9.71	0.0882572000083363\\
9.81	0.0860781219759974\\
9.91	0.0839528455719741\\
10.01	0.0818800424293303\\
};
\addlegendentry{Regime I};

\addplot [color=black,solid,forget plot]
  table[row sep=crcr]{%
0.0001	1.0000\\
0.1001	0.8579\\
0.2001	0.7978\\
0.3001	0.7517\\
0.4001	0.7129\\
0.5001	0.6791\\
0.6001	0.6488\\
0.7001	0.6212\\
0.8001	0.5959\\
0.9001	0.5724\\
1.0001	0.5505\\
1.1001	0.5300\\
1.2001	0.5106\\
1.3001	0.4924\\
1.4001	0.4751\\
1.5001	0.4587\\
1.6001	0.4431\\
1.7001	0.4282\\
1.8001	0.4140\\
1.9001	0.4004\\
2.0001	0.3874\\
2.1001	0.3750\\
2.2001	0.3630\\
2.3001	0.3516\\
2.4001	0.3406\\
2.5001	0.3300\\
2.6001	0.3198\\
2.7001	0.3100\\
2.8001	0.3005\\
2.9001	0.2914\\
3.0001	0.2827\\
3.1001	0.2742\\
3.2001	0.2660\\
3.3001	0.2582\\
3.4001	0.2505\\
3.5001	0.2432\\
3.6001	0.2361\\
3.7001	0.2292\\
3.8001	0.2225\\
3.9001	0.2161\\
4.0001	0.2099\\
4.0991	0.2039\\
4.1991	0.1981\\
4.2991	0.1924\\
4.3991	0.1870\\
4.4991	0.1817\\
4.5991	0.1765\\
4.6991	0.1715\\
4.7991	0.1667\\
4.8991	0.1621\\
4.9991	0.1575\\
5.0991	0.1531\\
5.1991	0.1489\\
5.2991	0.1447\\
5.3991	0.1407\\
5.4991	0.1368\\
5.5991	0.1330\\
5.6991	0.1294\\
5.7991	0.1258\\
5.8991	0.1224\\
5.9991	0.1190\\
6.0991	0.1158\\
6.1991	0.1126\\
6.2991	0.1095\\
6.3991	0.1066\\
6.4991	0.1037\\
6.5991	0.1008\\
6.6991	0.0981\\
6.7991	0.0955\\
6.8991	0.0929\\
6.9991	0.0904\\
7.0991	0.0879\\
7.1991	0.0856\\
7.2991	0.0833\\
7.3991	0.0810\\
7.4991	0.0789\\
7.5991	0.0768\\
7.6991	0.0747\\
7.7991	0.0727\\
7.8991	0.0708\\
7.9991	0.0689\\
8.0981	0.0671\\
8.1981	0.0653\\
8.2981	0.0636\\
8.3981	0.0619\\
8.4981	0.0602\\
8.5981	0.0586\\
8.6981	0.0571\\
8.7981	0.0556\\
8.8981	0.0541\\
8.9981	0.0527\\
9.0981	0.0513\\
9.1981	0.0499\\
9.2981	0.0486\\
9.3981	0.0474\\
9.4981	0.0461\\
9.5981	0.0449\\
9.6981	0.0437\\
9.7981	0.0426\\
9.8981	0.0415\\
9.9981	0.0404\\
9.9991	0.0404\\
10.0001	0.0404\\
10.0011	0.0403\\
};
\addlegendentry{Regime II};
\end{axis}
\end{tikzpicture}% 
\end{subfigure}%
\begin{subfigure}{.5\textwidth}
  \centering
\setlength
\figureheight{6.3cm} 
\setlength
\figurewidth{6cm}
\begin{tikzpicture}

\begin{axis}[%
width=\figurewidth,
height=0.933\figureheight,
at={(0\figurewidth,0\figureheight)},
scale only axis,
separate axis lines,
every outer x axis line/.append style={black},
every x tick label/.append style={font=\color{black}},
xmin=0,
xmax=5,
xtick={0, 1, 2, 3, 4, 5},
xlabel={$x$},
xmajorgrids,
every outer y axis line/.append style={black},
every y tick label/.append style={font=\color{black}},
ymin=0.3,
ymax=1,
ytick={  0, 0.1, 0.2, 0.3, 0.4, 0.5, 0.6, 0.7, 0.8, 0.9,   1},
ylabel={},
ymajorgrids,
axis background/.style={fill=white},
title={$\rho_1=$0.8, $\rho_2=$0.9},
legend style={legend cell align=left,align=left,draw=black}
]
\addplot [color=blue,dashdotted]
  table[row sep=crcr]{%
0.01	0.8919956\\
0.01	0.8919956\\
0.11	0.8354807\\
0.21	0.7973083\\
0.31	0.7670288\\
0.41	0.7413264\\
0.51	0.7185474\\
0.61	0.6980308\\
0.71	0.6792633\\
0.81	0.6618521\\
0.91	0.6456046\\
1.01	0.6303485\\
1.11	0.6159902\\
1.21	0.6023422\\
1.31	0.5893706\\
1.41	0.5769454\\
1.51	0.5650737\\
1.61	0.5536619\\
1.71	0.542782\\
1.81	0.5322822\\
1.91	0.5220936\\
2.01	0.5122693\\
2.11	0.5028307\\
2.21	0.4935978\\
2.31	0.484717\\
2.41	0.476111\\
2.51	0.467689\\
2.61	0.4595574\\
2.71	0.4515985\\
2.81	0.4438793\\
2.91	0.4363595\\
3.01	0.4290024\\
3.11	0.4218676\\
3.21	0.4148985\\
3.31	0.408117\\
3.41	0.4014728\\
3.51	0.3949958\\
3.61	0.3886803\\
3.71	0.3824889\\
3.81	0.3764451\\
3.91	0.3705882\\
4.01	0.3648223\\
4.11	0.3591158\\
4.21	0.3536072\\
4.31	0.3481848\\
4.41	0.3428788\\
4.51	0.3376804\\
4.61	0.3326094\\
4.71	0.327612\\
4.81	0.3226945\\
4.91	0.3179261\\
5.01	0.3131971\\
};
\addlegendentry{Simulations};

\addplot [color=red,loosely dashed]
  table[row sep=crcr]{%
0.01	0.998889505944279\\
0.01	0.998889505944279\\
0.11	0.987852165765209\\
0.21	0.976936783898347\\
0.31	0.966142012751823\\
0.41	0.955466519624129\\
0.51	0.944908986539592\\
0.61	0.934468110085655\\
0.71	0.924142601251964\\
0.81	0.913931185271228\\
0.91	0.903832601461837\\
1.01	0.893845603072222\\
1.11	0.883968957126935\\
1.21	0.874201444274425\\
1.31	0.864541858636502\\
1.41	0.85498900765946\\
1.51	0.845541711966847\\
1.61	0.836198805213863\\
1.71	0.826959133943362\\
1.81	0.817821557443453\\
1.91	0.808784947606665\\
2.01	0.799848188790676\\
2.11	0.791010177680577\\
2.21	0.782269823152662\\
2.31	0.773626046139713\\
2.41	0.765077779497787\\
2.51	0.756623967874466\\
2.61	0.748263567578565\\
2.71	0.739995546451281\\
2.81	0.731818883738763\\
2.91	0.723732569966094\\
3.01	0.715735606812659\\
3.11	0.707827006988901\\
3.21	0.700005794114426\\
3.31	0.692271002597462\\
3.41	0.684621677515654\\
3.51	0.677056874498165\\
3.61	0.669575659609086\\
3.71	0.662177109232141\\
3.81	0.65486030995665\\
3.91	0.647624358464767\\
4.01	0.640468361419958\\
4.11	0.633391435356708\\
4.21	0.626392706571453\\
4.31	0.619471311014712\\
4.41	0.612626394184416\\
4.51	0.60585711102041\\
4.61	0.599162625800126\\
4.71	0.592542112035403\\
4.81	0.585994752370455\\
4.91	0.579519738480957\\
5.01	0.573116270974254\\
};
\addlegendentry{Regime I};

\addplot [color=black,solid,forget plot]
  table[row sep=crcr]{%
0.0001	1\\
0.1001	0.892873094\\
0.2001	0.848435567\\
0.3001	0.814457616\\
0.4001	0.785957676\\
0.5001	0.760997027\\
0.6001	0.738576949\\
0.7001	0.718101538\\
0.8001	0.699180812\\
0.9001	0.681542675\\
1.0001	0.664988115\\
1.1001	0.649366206\\
1.2001	0.634559145\\
1.3001	0.620472811\\
1.4001	0.607030524\\
1.5001	0.594168788\\
1.6001	0.581834275\\
1.7001	0.569981658\\
1.8001	0.558571998\\
1.9001	0.547571528\\
2.0001	0.536950726\\
2.1001	0.526683589\\
2.2001	0.516747059\\
2.3001	0.507120569\\
2.4001	0.497785669\\
2.5001	0.488725732\\
2.6001	0.479925700\\
2.7001	0.471371878\\
2.8001	0.463051769\\
2.9001	0.454953921\\
3.0001	0.447067809\\
3.1001	0.439383730\\
3.2001	0.431892711\\
3.3001	0.424586432\\
3.4001	0.417457162\\
3.5001	0.410497698\\
3.6001	0.403701313\\
3.7001	0.397061715\\
3.8001	0.390573004\\
3.9001	0.384229638\\
4.0001	0.378026405\\
4.0961	0.372198573\\
4.1961	0.366256002\\
4.2961	0.360439797\\
4.3961	0.354745801\\
4.4961	0.349170069\\
4.5961	0.343708860\\
4.6961	0.338358618\\
4.7961	0.333115960\\
4.8961	0.327977663\\
4.9961	0.322940656\\
5.0091	0.322293130\\
};
\addlegendentry{Regime II};

\end{axis}
\end{tikzpicture}%
\end{subfigure}
\caption{\textit{In the above two plots we see that again the Regime II approximation works significantly better than the Regime I approximation. Even for $\rho_1=0.6$ and $\rho_2=0.8$, the Regime II approximation shows remarkably good fit.}}
\label{fig:HT2}
\end{figure}

{\small
\begin{table}[ht!]
\centering
{\small
\begin{tabular}{r|c|c|c||c|c|c|c|}
 & \multicolumn{3}{|c||}{Fig.\ \ref{fig:HT1}, left plot}&  \multicolumn{4}{|c|}{Fig.\ \ref{fig:HT1}, right plot}\\
\hline
$x$&\textbf{Simul}&\textbf{R1}&\textbf{R2}&$x$&\textbf{Simul} &\textbf{R1}&\textbf{R2}\\
1&0.975&0.990&0.984&1&0.884&0.990&0.888\\
20&0.800&0.817&0.810&5&0.750&0.951&0.753\\
40&0.653&0.668&0.662&10&0.654&0.904&0.656\\
80&0.436&0.446&0.442&15&0.584&0.859&0.585\\
100&0.356&0.364&0.362&20&0.527&0.817&0.528\\
150&0.215&0.220&0.219&25&0.480&0.777&0.480\\
200&0.129&0.133&0.132&30&0.439&0.739&0.439\\
250&0.077&0.080&0.080&35&0.403&0.702&0.403\\
300&0.047&0.048&0.049&40&0.372&0.668&0.372\\
400&0.017&0.018&0.019&45&0.344&0.635&0.343\\
500&0.006&0.006&0.008&50&0.318&0.603&0.318
\end{tabular}}
\caption{\textit{The values in this table correspond to the left and right plot in Figure \ref{fig:HT1}. {\em {\bf Simul}} stands for the simulated values, and {\em {\bf R1}}, {\em {\bf R2}} stand for the approximated values using Regime I, II, respectively. }}
\label{tab:HT1}
\end{table}

\begin{table}[ht!]
\centering
{\small
\begin{tabular}{r|c|c|c||c|c|c|c|}
 & \multicolumn{3}{|c||}{Fig.\ \ref{fig:HT2}, left plot}&  \multicolumn{4}{|c|}{Fig.\ \ref{fig:HT2}, right plot}\\
\hline
$x$&\textbf{Simul}&\textbf{R1}&\textbf{R2}&$x$&\textbf{Simul} &\textbf{R1}&\textbf{R2}\\
1&0.498&0.777&0.551&0.5&0.719&0.945&0.761\\
2&0.362&0.605&0.387&1.0&0.630&0.894&0.665\\
3&0.273&0.471&0.283&1.5&0.565&0.846&0.594\\
4&0.210&0.367&0.210&2.0&0.512&0.800&0.537\\
5&0.163&0.286&0.157&2.5&0.468&0.757&0.489\\
6&0.128&0.223&0.119&3.0&0.429&0.716&0.447\\
7&0.101&0.173&0.090&3.5&0.395&0.677&0.410\\
8&0.080&0.135&0.069&4.0&0.365&0.640&0.378\\
9&0.064&0.105&0.053&4.5&0.338&0.606&0.349\\
10&0.051&0.082&0.040&5.0&0.313&0.573&0.323
\end{tabular}}
\caption{\textit{The values in this table correspond to the left and right plot in Figure \ref{fig:HT2}. The abbreviations are as in Table \ref{tab:HT1}.}}
\label{tab:HT2}
\end{table}}

In addition, we  estimated the probabilities by simulation. The results are gathered in Tables \ref{tab:HT1} and \ref{tab:HT2}, and are plotted in Figs.\ \ref{fig:HT1} and \ref{fig:HT2}. Observe from Fig.\ \ref{fig:HT1} that the Regime II approximation is substantially more accurate than the Regime I approximation when $\rho_2$ is high (in this case $\rho_2=0.99$). By comparing the two plots in Fig.\ \ref{fig:HT1}, we see that increasing $\rho_1$ negatively affects the performance of the Regime I   approximation.  Figure \ref{fig:HT2} shows that the Regime II approximation  works remarkably well even when relatively low loads are imposed on both servers.
Our experiments reveal that it is only reasonable to use Regime~I approximations in a tandem queue when the load of the first server $\rho_1$ is low; in all other cases it is outperformed by the Regime II approximation. If $\rho_1$ is high, then there is a stronger dependence between the up- and downstream workloads (cf. Eqn.\ \eqref{eq:cg}, noting that $\rho_1$ increases as $\g$ decreases). Apparently, the dependence between both workloads, which is ignored in Regime I, has a crucial impact.
\end{example}

\iffalse Intuitively it may not be clear why using a heavy-traffic approximation on both queues is more accurate than using a heavy-traffic approximation on only one queue. Indeed, we use two approximations instead of only one! However, the Regime II approach allows for dependence between the two workloads. It turns out that this improves the approximation substantially, especially when the load in the first queue is large as well. In this case, the correlation between the queue is larger (recall Equation \eqref{eq:corr}).
\fi

\section{Heavy-tailed input}
\label{sec:infvar}
In this section we consider spectrally-positive L\'evy input processes with {increments with infinite variance}. Unlike in the finite variance case, the precise form of the heavy-traffic limit depends on the specific features of the L\'evy input process. In Section \ref{sec:CPHT} we consider compound Poisson input with heavy-tailed jobs, and in Section \ref{sec:alpha} we consider $\a$-stable L\'evy input (where $1<\a<2$). Note that $\a$-stable L\'evy motion can be regarded as a generalization of Brownian motion. Indeed, for $\a=2$, an $\a$-stable L\'evy motion reduces to a Brownian motion.

\begin{remark}
We only consider Regime I results, because we have not managed to compute Regime II results here. In the finite variance case, we relied on the existence of the inverse function of $\phi$ in the Brownian case, to construct the $\psi(s\e(r_1-r_2))$ as in Lemma \ref{lem:psiK}. However, for heavy-tailed input, there is in general no inverse function of $\phi$ available, except for some special cases, such as $\f32$-stable L\'evy motion. 
\end{remark}

Before we state the result, we introduce some notation. We denote 
\[
E_{\alpha}(z) = \sum_{n=0}^\infty \f{z^n}{\Gamma(\alpha n + 1)}
\]
for the Mittag-Leffler function with parameter $\a$. Random variables that have a distribution function $1-E_\a(x)$ are called {\it Mittag-Leffler distributed} with parameter $\a$. {Suppose that $M$ is Mittag-Leffler distributed with parameter $\a$, then the LST of $M$ is given by
\[
\E e^{- sM} = \f{1}{1+ s^{\a}}.
\]
}
Furthermore, suppose a measurable function $L$ defined on some neighborhood $[X,\infty)$ of $\infty$, satisfies
\[
\lim_{x\to\infty} \f{L(ax)}{L(x)} = 1,\quad \forall a>0,
\]
then it is called a slowly varying function \cite{RV}. For notational brevity we sometimes write $f(x)\sim g(x)$ (as $x\to\infty$), to denote $\lim_{x\to\infty} f(x)/g(x) =1$, for generic functions $f,g$.

{
\subsection{Compound Poisson}
\label{sec:CPHT}
In this section we consider spectrally positive compound Poisson input processes with heavy-tailed jobs.}

\begin{remark}
In \cite{Boxma1999} a heavy-traffic problem for heavy-tailed input was studied in a GI/G/1 setting. In their paper, the correct scaling function $\De$ was also found by letting it be the zero of an appropriate equation. We follow a similar approach.
\end{remark}

\begin{proposition}
\label{prop:ML}
Let the input process $J\in\mathcal{S}^+$ to the first queue be a compound Poisson process with heavy-tailed service requirements, that is, the {distribution of} the service requirement $B$ satisfies
\begin{equation}
\label{eq:BRV}
\Pb(B>x) \sim x^{-\nu} L(x), \quad \text{as} \quad x\to\infty,
\end{equation}
where $L$ is some slowly varying function. Suppose that the load of the first queue is fixed and the load of the second queue is increasing to one as $\e\downarrow0$. For $\e>0$ small enough, there is a unique solution $s=\De$ to
\[-\l\Gamma(1-\nu)\f{(r-\e)^\nu}{r^{\nu+1}} s_2^{\nu-1} L(1/s_2) = \e,\]
such that $\De\downarrow0$. It holds that 
\[
\lim_{\e\downarrow0} \E e^{-s_1 Q^{(1)}-s_2\De Q^{(2)}} = \f{r s_1}{\phi(s_1)}\cdot \f{1}{1+s_2^{\nu-1}}.
\]
\end{proposition}
\begin{proof}
Suppose the input process $J$ is of the compound Poisson type. More precisely, we have a Poisson process $N$ with rate $\lambda$ and we assume
\[
J_t = \sum_{k=1}^{N(t)} B_k, \text{ where the $B_k$ are i.i.d., independent of $N(t),$ and such that } \E J_1 = 1.
\]
Then the cumulative net input processes for the first server and the whole system ($i=1,2$ respectively) are defined by
\[
X_t^{(i)} = \sum_{k=1}^{N(t)} B_k - r_i t.
\]

Suppose we have a compound Poisson input process, then  
$
\phi(s) = sr_1 - \l + \l b(s),
$
where $b(s) = \E e^{-s B}$ (cf.\ Eqn.\ \eqref{eq:laplace exponent}). Suppose the service time $B$ is regularly varying, with index $1<\nu<2$. Then it takes the form of Eqn.\ \eqref{eq:BRV}. 
By applying Thm.\ \ref{thm:RV},  
\[
b(s) - 1 - c_1 s \sim  - \Gamma(1-\nu)  s^\nu L(1/s) \quad\text{as}\quad s\downarrow0,
\]
with $1<\nu<2$. Substitution yields
$\phi(s) \sim (\l c_1 + r_1)s - \l\Gamma(1-\nu) s^\nu L(1/s)$. We assumed $\l \E B=1$, so $b'(0) = -\f1\l = c_1$. Recall that $r_1-1=r$, and therefore
\[
\phi(s) - rs \sim  - \l \Gamma(1-\nu) s^\nu L(1/s). 
\]
By Lemma 9.2 from \cite{Mandjes} (see also Lemma \ref{lem:9.2} in Appendix \ref{app:rv}), we find
\begin{equation}
\label{eq:heavytailpsi}
\psi(s)-\f{1}{r}s \sim   \l \Gamma(1-\nu) \f{1}{r^{\nu+1}} s^\nu L(1/s)\quad\text{as}\quad s\downarrow0.
\end{equation}
We now {identify} a scaling function $\Delta(\e)$ such that we have convergence to a non-degenerate distribution. By making use of Eqn.\ \eqref{eq:heavytailpsi} and by scaling the workload of the downstream queue by a function $\Delta(\e)$, for which $\De\downarrow0$ as $\e\downarrow0$, we obtain
\begin{align}
\label{eq:lstde}
\E e^{-s_1Q^{(1)}-s_2\De Q^{(2)}} &\sim 
\f{1}{1 + \f1\e C (r-\e)^\nu s_2^{\nu-1}\De^{\nu-1} L(\f{1}{s_2\De})} \\
&\cdot \f{(r-\e) s_2\De - C s_2^\nu (r-\e)^\nu \De^\nu L(\f{1}{s_2\De(r-\e)})- rs_1}{(r-\e)s_2\De -\phi(s_1)},\nonumber
\end{align}
where $C:=-\l\Gamma(1-\nu){r^{-\nu-1}}$, for $s_1,s_2\geq0$ fixed and $\e\downarrow0$. Consider the equation
\begin{equation}
\label{eq:no3}
C(r-\e)^\nu s_2^{\nu-1} L(1/s_2) = \e.
\end{equation}
We will show that this equation has a unique zero for $\e$ close enough to zero, and we call the zero $\Delta(\e)$. Indeed, by Thm.\ 1.5.4 in \cite{RV}, we have that 
\[
C(r-\e)^{\nu} s^{\nu-1} L(1/s) \sim \xi(1/s),\quad s\downarrow0,
\]
where $s\mapsto\xi(s)$ is a non-decreasing function (hence $s\mapsto \xi(1/s)$ non-increasing). So if $\e$ is chosen small enough, the $s$ solving Eqn.\ \eqref{eq:no3} also becomes small and
\[
\f{C(r-\e)^\nu s^{\nu-1}L(1/s)}{\xi(1/s)} \xi(1/s) \approx \xi(1/s),
\]
so the left-hand side of Eqn.\ \eqref{eq:no3} is asymptotically monotone. This ensures that there is exactly one root $\Delta(\e)$ for all $\e>0$ small enough. Moreover, note that $\De$ indeed satisfies $\De\downarrow0$ as $\e\downarrow0$.

Therefore, we have
\begin{equation}
\label{eq:nr1}
C(r-\e)^\nu \De^{\nu-1} L(1/(s_2\De)) = \e.
\end{equation}
Now consider the first factor on the right-hand side in Eqn. \eqref{eq:lstde}. {Substituting} Eqn.\ \eqref{eq:nr1} into this factor,  
\[
\f{1}{1 + \f1\e C (r-\e)^\nu s_2^{\nu-1}\De^{\nu-1} L(\f{1}{s_2\De})} = \f{1}{1+s_2^{\nu-1} \f{L(1/(s_2\De))}{L(1/s_2)}}\stackrel{\e\downarrow0}{\longrightarrow}\f{1}{1+s_2^{\nu-1}},
\]
where we make use of the fact that $L$ is  slowly varying  at $\infty$. Now consider the following part of the second factor in Eqn.\ \eqref{eq:lstde}:
\begin{align*} 
C s_2^\nu (r-\e)^\nu \De^\nu L\Big(\f{1}{s_2\De(r-\e)}\Big) &= \left[ C(r-\e)^\nu \De^{\nu-1} L(\f{1}{s_2\De})\right] \De \f{L\Big(\f{1}{s_2\De(r-\e)}\Big)}{L\Big(\f{1}{s_2\De}\Big)}\\
&=\e\De\f{L\Big(\f{1}{s_2\De(r-\e)}\Big)}{L\Big(\f{1}{s_2\De}\Big)} \sim \e\De,
\end{align*}
where we substituted the part between square brackets by making use of Eqn.\ \eqref{eq:nr1}, and used that $L$ is slowly varying. By again exploiting the fact that $\De\downarrow0$ as $\e\downarrow0$, the result now follows from Eqn.\ \eqref{eq:lstde}.
\end{proof}

\begin{example}
\label{rem:HT}
Suppose that we are in the setting of Prop.\ \ref{prop:ML}, but {we are in the special case} that $\lim_{x\to\infty}L(x) = L\in\R$. Then {a} correct scaling function is
\[\De = \left(\f{\e}{\f{\l}{r} \Gamma(1-\nu) L}\right)^{\f{1}{\nu-1}}.\]

Prop.\ \ref{prop:ML} can be used to find a heavy-traffic approximation as follows. We have 
\[
\lim_{\e\downarrow0} \E e^{- s \De Q^{(2)}} = \f{1}{1+ s^{\nu-1}},
\]
so that,  for $x\geq0$, and  $\e>0$ small,
\[
\Pb(\De Q^{(2)}> x ) \approx E_{\nu-1}(-x^{\nu-1}).
\]
By substitution we {thus} obtain the heavy-traffic approximation for $x\geq0$, and  $\e>0$ small, \[
\Pb(Q^{(2)}> x) \approx E_{\nu-1}\Big(-{ \De^{\nu-1} x^{\nu-1}}\Big).
\]
\end{example}

\subsection{$\a$-stable L\'evy motion}
\label{sec:alpha}
In this subsection we prove the following result. {It entails that the workloads are asymptotically independent in the heavy-traffic limit, and that the marginals correspond to scaled Mittag-Leffler distributed random variables.}

\begin{proposition}
\label{prop:HTR1}
Let the input process $J\in\mathcal{S}^+$ to the first queue be a spectrally positive $\a$-stable L\'evy motion, with $1<\a<2$. Suppose that the load of the first queue is fixed and the load of the second queue is increasing as $\e\downarrow0$ and scaled by $\e^{\beta}$, with  $\beta:=({\a-1})^{-1}$. 
It holds that \[
\lim_{\e\downarrow0} \E e^{-s_1(C/r)^{\beta} Q^{(1)}-s_2(\e C)^{\beta}  Q^{(2)}} = \f{1}{1+s_1^{\a-1}}\cdot \f{1}{1+s_2^{\a-1}},
\]
with $C:= ({\cos(\pi(\f\a 2 -1))})^{-1}$. 
\end{proposition}

\iffalse
\begin{remark}
Note that for $\a=2$, it follows that $C=1$ and we have a Brownian motion with variance $\s^2=2$ (because an $\a$-stable distribution is a normal distribution for $\a=2$, with variance twice the scale parameter). Then Proposition \ref{prop:HTR1} corresponds to the findings in Section \ref{sec:BM}.
\end{remark}\fi

\begin{proof}[Proof of Proposition \ref{prop:HTR1}]
The Laplace exponent is given by
$
\phi(s) = (r_1-1)s + Cs^\a$. It follows by Lemma 9.2 from \cite{Mandjes}, that
\[
\psi(s) = c_1 + c_2 s - \f{C}{r_1-1} \Big(\f{s}{r_1-1}\Big)^\a + o(s^\a).
\]
We know that $\psi(0)=0$, hence $c_1=0$, and $\psi'(0) = \f{1}{\phi'(0)} = \f{1}{r_1-1}$, hence $c_2=\f{1}{r_1-1}$. This leads to
\[
\psi(s) = \f{s}{r_1-1}-\f{C}{r_1-1}\Big(\f{s}{r_1-1}\Big)^\a+ o(s^\a). 
\]

    %Then
    %\begin{align*}
    %\psi(s\e^{\f{1}{\a-1}}(r-\e)) &\approx \a\e^{\f{1}{\a-1}} - \f s %r\e^{\f{1}{\a-1}}-\f C r s^\a \e^{\f{\a}{\a-1}}=s\e^{\f{1}{\a-1}} + %\e^{\f{\a}{\a-1}}\Big(-\f\a r - \f C r s^\a\Big).
    %\end{align*}
    %Using Corollary \ref{thm:marginal}, we obtain
    %\[
    %\E e^{-s\e^\f{1}{\a-1} Q^{(2)}} = \f{\e}{r-\e}\f{s\e^{\f{1}{\a-1}}(1+o(\e))}{s\e^{\f{1}{\a-1}}-s\e^{\f{1}{\a-1}}+\e^{\f{\a}{\a-1}}\Big(\f s r + \f Cr s^\a\Big)(1+o(\e))}.
    %\]
    %Now we are ready to take the limit
    %\[
    %\lim_{\e\downarrow0} \E e^{-s \e^{\f{1}{\a-1}}Q^{(2)}}=\f{1}{1+Cs^{\a-1}}.
    %\]
It follows from Thm.\ \ref{thm:joint}, that
\begin{align*} 
\E e^{-s_1 Q^{(1)} - \e^{\f{1}{\a-1}}s_2 Q^{(2)}}  =& 
\f{\e s_2 \e^{\f{1}{\a-1}}(1+o(s^\a\e^{\f{\a}{\a-1}}))}{\e^{\f{1}{\a-1}}s_2 - \e^{\f{1}{\a-1}}\f{s_2(r-\e)}{r} + \f Cr\Big(\f{s_2(r-\e)}{r}\Big)^\a \e^{\f{\a}{\a-1}}}\\
&\cdot\f{\f1r\e^{\f{1}{\a-1}}s_2(r-\e) - \f Cr\Big(\f{s_2(r-\e)}{r}\Big)^\a \e^{\f{\a}{\a-1}}-s_1}{\e^{\f{1}{\a-1}}(r-\e)s_2-rs_1-Cs_1^\a}. 
\end{align*}
Consequently,
\begin{equation}
\label{eq:astable}
\lim_{\e\downarrow0} \E e^{-s_1 Q^{(1)} - \e^{\f{1}{\a-1}}s_2 Q^{(2)}} = \f{r}{r+C s_1^{\a-1}}\f{1}{1+Cs_2^{\a-1}},
\end{equation}
{which implies the claim.}
\end{proof}

\iffalse
\begin{remark}
Consider Equation \eqref{eq:astable}. As a sanity check, we can verify that Corollary \ref{cor:r1toinfty} holds. That is, the second queue behaves as the first queue with its own (limiting) rate 1 as $r_1\to\infty$.
\end{remark}\fi

In case $\a=\f32$, $\psi$ can be calculated explicitly and the result can be obtained without the use of Tauberian theorems.  {We include it in the paper, as the calculations potentially contain clues as to how Regime II results can be eventually obtained}. \iffalse Also, the calculations provide some insight in details that are not visible when Tauberian methods are applied. \fi

\begin{example}[Explicit calculations for $\a=\f32$]
We assume a $\f32$-stable input process, so that the Laplace exponent is given by
\[
\phi(s) = r s + \f{1}{\cos(\pi(\f\a 2 - 1))} s^{\f32} = r s + \sqrt{2} s\sqrt{s}.
\]
Define
\begin{align} 
\label{eq:R}
R(s)&:=-\frac{r^3}{54 \sqrt{2}}+\sqrt{\f18 s^2 - \f{s r^3}{108}}+\frac{s}{2 \sqrt{2}}.
\end{align}
By making a substitution $s^2\leftarrow s$, $\phi$ turns into a third order polynomial., which can be inverted by using Cardano's formula. It follows that the inverse function of $\phi$ is given by
\begin{equation}
\label{eq:cardano}
\psi(s)=\left(R(s)^{\f13} + \f{r^2}{18R(s)^{\f13}} - \f{r}{3\sqrt{2}}\right)^2.
\end{equation}

Note that $s = \phi(\psi(s)) = r \psi(s) + \sq2 \psi(s)^{\f32}$. Define the function $\z$ such that  $\z(s)^2 = \psi(s)$. Then
$\psi(s) = \z(s)^2 = r^{-1}({s-\sq2 \z(s))^3}$. So 
\begin{equation} 
\label{eq:psizeta}
\psi( s\e^2(r-\e)) = \f1r  s\e^2(r-\e) - \f{\sq2}{r}\z( s\e^2(r-\e))^3.
\end{equation}
Now we can focus on 
\begin{equation}
\label{eq:zeta}
\z( s\e^2(r-\e)) = R( s\e^2(r-\e))^{\f13} + \f{r^2}{18R( s\e^2(r-\e))^{\f13}} - \f{r}{3\sqrt{2}},
\end{equation}
where we can simplify by constructing the Taylor series of $R^{\f13}$. First note that
\[
R( s\e^2(r-\e)) = -\f{r^3}{54\sq2} + \sq{\f18  s^2\e^4(r-\e)^2 - \f{ s\e^2(r-\e) r^3}{108}} + o(\e).
\]
By rewriting this and using Taylor expansions for the square roots, neglecting all terms of smaller order than $\e$, we obtain
\begin{align*} 
R( s\e^2(r-\e)) = -\f{r^3}{54\sq2} +\sq{\f{r^4 s}{108}}\e i \sq{1+o(\e)} +o(\e) = - \f{r^3}{54\sq2}(1-g\e)+o(\e),
\end{align*}
where $i$ denotes $\sq{-1}$ and we defined $g:=\f{3\sq6}{r}\sq{ s}i$. Again using a Taylor expansion, we find
\begin{align*} 
R( s\e^2(r-\e))^{\f13} &= (-1)^{\f13}\f{r}{3\sq2}\Big(1-\f13 g\e\Big)+ o(\e),\\
R( s\e^2(r-\e))^{-\f13} &= (-1)^{-\f13}\f{3\sq2}{r}\Big(1+\f13 g\e\Big)+ o(\e).
\end{align*}

Substituting this into Eqn.\ \eqref{eq:zeta}  yields
\begin{align*}
\zeta( s\e^2(r-\e)) &= \f{rg\e}{9\sq2}\Big(-(-1)^{\f13}+(-1)^{-\f13}\Big)=\f{rg\e}{9\sq2}(-\sq3 i)+ o(\e),
\end{align*}
by making use of $(-1)^{\f13} = e^{i\pi/3} = \f12 + \f12\sq3 i$ and $(-1)^{-\f13} = e^{-i\pi/3} = \f12 - \f12\sq3 i$. Recalling the definition of $g$, we find $\z( s\e^2(r-\e)) = \e\sq s + o(\e)$. Substituting this into Eqn.\ \eqref{eq:psizeta}, yields $\psi( s\e^2(r-\e)) =  s\e^2-(1+\sq{2 s})\f{ s\e^3}{r} + o(\e^3)$. It can be verified that terms of smaller magnitudes do not contribute to the heavy-traffic {version} of Corollary \ref{thm:marginal}. Using this corollary yields
\begin{align*}
\lim_{\e\downarrow0} \E e^{- s \e^2 Q^{(2)}} = \lim_{\e\downarrow0} \f{1-(1+\sq{2 s})\f{\e}{r} + o(\e)}{1+\sq{2 s} + o(1)} = \f{1}{1+\sq{2s}},
\end{align*}
which corresponds to Eqn.\ \eqref{eq:astable} with $\a=\f32$ (and here we considered $s_1=0$).
\end{example}

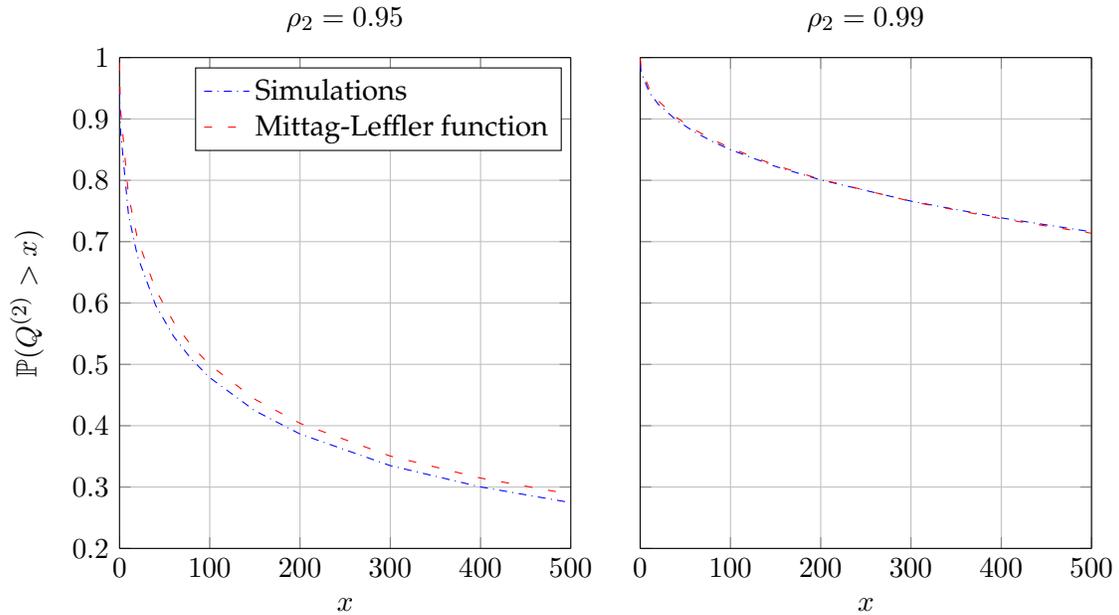
\begin{figure}[b!]
\centering
\begin{subfigure}{.5\textwidth}
  \centering
\setlength
\figureheight{7cm} 
\setlength
\figurewidth{6cm} 
\begin{tikzpicture}

\begin{axis}[%
width=\figurewidth,
height=0.933\figureheight,
at={(0\figurewidth,0\figureheight)},
scale only axis,
separate axis lines,
every outer x axis line/.append style={black},
every x tick label/.append style={font=\color{black}},
xmin=0,
xmax=500,
xtick={  0,   100,  200,  300,  400,  500},
xlabel={$x$},
xmajorgrids,
every outer y axis line/.append style={black},
every y tick label/.append style={font=\color{black}},
ymin=0.2,
ymax=1,
ytick={  0, 0.1,0.2,.3,.4,.5,.6,.7,.8,.9,1},
ylabel={$\Pb(Q^{(2)}>x)$},
ymajorgrids,
axis background/.style={fill=white},
title={$\rho_2=0.95$},
legend style={legend cell align=left,align=left,draw=black}
]
\addplot [color=blue,dashdotted]
  table[row sep=crcr]{%
0.01	0.948619065\\
1	0.883784084583333\\
10	0.743574465069444\\
20	0.676030963680556\\
40	0.596509844930556\\
60	0.545522366527778\\
80	0.507968260208333\\
100	0.478380575902778\\
150	0.424370277638889\\
200	0.386535795972222\\
300	0.334975131597222\\
400	0.300250771319444\\
500	0.274610677152778\\
};
\addlegendentry{Simulations};

\addplot [color=red,loosely dashed]
  table[row sep=crcr]{%
0.01	0.991353981044732\\
1	0.918571360936812\\
10	0.774854933245452\\
20	0.704685169136382\\
40	0.621933689399854\\
60	0.568886445056035\\
80	0.529864249073697\\
100	0.499139727756242\\
150	0.443130520297336\\
200	0.403922720758789\\
300	0.350568583520147\\
400	0.31468843732451\\
500	0.288266828143892\\
};
\addlegendentry{Mittag-Leffler function};

\end{axis}
\end{tikzpicture}%

 %Kenmerk: 111
\end{subfigure}%
\begin{subfigure}{.5\textwidth}
  \centering
\setlength
\figureheight{7cm} 
\setlength
\figurewidth{6cm} 
%\vspace{-0.45cm}
\begin{tikzpicture}

\begin{axis}[%
width=\figurewidth,
height=0.933\figureheight,
at={(0\figurewidth,0\figureheight)},
scale only axis,
separate axis lines,
every outer x axis line/.append style={black},
every x tick label/.append style={font=\color{black}},
xmin=0,
xmax=500,
xtick={  0, 100, 200,  300,  400,  500},
xlabel={$x$},
xmajorgrids,
every outer y axis line/.append style={black},
every y tick label/.append style={font=\color{white}},
ymin=0.2,
ymax=1,
ytick={  0, .1,.2,.3,.4,.5,.6,.7,.8,.9,1},
ylabel={},
ymajorgrids,
axis background/.style={fill=white},
title={$\rho_2=0.99$},
legend style={at={(0.9,0.2)},anchor=south east}
]
\addplot [color=blue,dashdotted]
  table[row sep=crcr]{%
0.01	0.9900711225\\
1	0.97684440325\\
10	0.943252485583333\\
20	0.9240350865\\
40	0.89800458775\\
60	0.878864200416667\\
80	0.863258550083333\\
100	0.849874396333333\\
150	0.822537581\\
200	0.800634638916667\\
300	0.766041672\\
400	0.738739646083333\\
500	0.7160614055\\
};
%\addlegendentry{Simulations};

\addplot [color=red,loosely dashed]
  table[row sep=crcr]{%
0.01	0.998331496155076\\
1	0.983509867594701\\
10	0.949285369809861\\
20	0.929459595049025\\
40	0.902523610018661\\
60	0.882688030581952\\
80	0.866501600264381\\
100	0.852631635170353\\
150	0.824267321987494\\
200	0.801519175583235\\
300	0.765577221132344\\
400	0.737273381043008\\
500	0.713731148920678\\
};
%\addlegendentry{Mittag-Leffler function};

\end{axis}
\end{tikzpicture}%
\end{subfigure}
\caption{\textit{Using the Mittag-Leffler function as an approximation. We simulated 288 sample paths each consisting of $50\cdot 10^6$ arrivals of Pareto distributed jobs. In both cases we used $\l=1$, $\nu=1.5$, $\rho_1=\f12$, and we only varied $\rho_2$ as indicated above the plots.}}
\label{fig:HTsimul}
\end{figure}

\begin{table}[ht!]
\centering
{\small
\begin{tabular}{r||c|c|c||c|c|c}
  & \multicolumn{3}{c||}{$\rho_2=0.95$} & \multicolumn{3}{c}{$\rho_2=0.99$}\\
$x$ & \textbf{Simul} & \textbf{M-L} & \textbf{diff} (\%)   & \textbf{Simul} & \textbf{M-L} & \textbf{diff} (\%)\\
\hline
10&0.744&0.775&4.2&0.943&0.949&0.64\\
20&0.676&0.705&4.3&0.924&0.929&0.54\\
40&0.597&0.622&4.2&0.900&0.903&0.33\\
60&0.546&0.569&4.2&0.879&0.883&0.46\\
80&0.508&0.530&4.3&0.863&0.867&0.46\\
100&0.478&0.499&4.4&0.850&0.853&0.35\\
150&0.424&0.443&4.5&0.823&0.824&0.12\\
200&0.387&0.404&4.4&0.801&0.802&0.12\\
300&0.335&0.351&4.8&0.766&0.766&0.00\\
400&0.300&0.315&5.0&0.739&0.737&-0.27\\
500&0.274&0.288&5.1&0.716&0.714&-0.28
\end{tabular}}
\caption{\textit{The $x$ indicates the size of the workload. The columns {\em {\bf Simul}} and {\em {\bf M-L}} show the probabilities that $\Pb(Q^{(2)}>x)$, for the simulated sample paths and the heavy-traffic approximation from Example \ref{rem:HT}, respectively. The last column shows the relative difference between the two values, that is, {\em {\bf diff}} equals ({\em {\bf M-L}} $-$ {\em {\bf Simul}})/ {\em {\bf Simul}}$\,\cdot 100\%$. }}
\label{tab:plot}
\end{table}

\subsection{Numerical heavy-traffic approximations}
Suppose the {tandem system is fed by} a compound Poisson input process {with} jobs that are {Pareto} distributed. In this case, the slowly varying function from Prop.\ \ref{prop:ML} is actually a constant. In Example \ref{rem:HT}, we obtained the corresponding heavy-traffic approximation.  {Fig.\ \ref{fig:HTsimul} facilitates a comparison  between estimates obtained from simulations and the Mittag-Leffler (Regime I) heavy-traffic approximation}. As expected, we see that as $\rho_2$ increases, the heavy-traffic approximation becomes more accurate, by comparing the left plot (where $\rho_2=0.95$) to the right plot (where $\rho_2=0.99$). We show the plotted values in Table \ref{tab:plot}, along with the relative difference between the two values.

\section{Discussion and concluding remarks}
In this paper we considered two types of heavy-traffic regimes for a two-node fluid tandem queue with {spectrally-positive} L\'evy input. In Regime I, only the second server experiences heavy traffic. In this case, the load of the first server has no influence on the steady-state distribution of the workload in the second server. In Regime II, where both servers experience heavy traffic, the dependence structure between both workloads is {preserved}. {In case the increments of the L\'evy input process have finite variance, we have obtained Regime I and II results, whereas for the infinite variance case we established Regime I results.}

The numerical {experiments} led to the  interesting insight  that (for finite-variance input processes) the Regime II approximation performs typically better than the Regime I approximation, particularly when the load of the first server is high as well. This leads us to wonder if results of this kind {carry over to} a more general setting. More specifically, when considering a more general fluid tandem network (e.g., a $n$-node tandem system) with multiple `bottlenecks', does Regime II provide better approximations than Regime I?

\vb

An open problem concerns Regime II results in case {the increments of the input process have infinite variance}. It is not clear how such results can be established. In the finite variance case we could define an inverse Laplace exponent that was in line with the exact inverse for Brownian motion.
However, in the case of heavy-tailed input, e.g.\ for $\a$-stable L\'evy motion, there is no explicit inverse Laplace exponent for all $1<\a<2$, and hence a fundamentally different approach {needs to be developed}.

Another direction for further research concerns weak convergence results at the path level (rather than the stationary workload that was considered in the present paper).

\section*{Acknowledgements}
The research for this paper is partly funded by the NWO Gravitation project NETWORKS, grant number 024.002.003. The research of Onno Boxma was also partly funded by the Belgian Government, via the IAP Bestcom project.

\begin{appendix}
\section{Useful Tauberian results}
\label{app:rv}
We turn to a law $F$ defined on $[0,\infty)$. We study the LST $\hat{F}$. Write
\[
\mu_n := \E X^n = \int_{[0,\infty)} x^n dF(x) \quad (n=0,1,\ldots)
\]
for the $n$-th moment. When $\mu_n<\infty$, $\hat{F}(s)$ may be expanded in a Taylor series as far as the $s^n$ term:
\[
\hat{F}(s) = \sum_{r=0}^n \mu_r(-s)^r / r! \quad (s\downarrow0).
\]
To relate the tail behaviour of $F$ to the behavior of $\hat{F}$ at zero, one needs to eliminate the polynomial $\sum_{r=0}^n \mu_r(-s)^r/r!$, which can be done by subtraction or differentiation. This leads to the following definitions:
\[
f_n(s) := (-1)^{n+1} \left(\hat{F}(s) - \sum_{r=0}^n \mu_r(-s)^r/r! \right),
\]
\[
g_n(s) := \f{d^n f_n(s)}{ds^n} = \mu_n - (-1)^n \hat{F}^{(n)}(s),
\]
thus $f_0(s) = g_0(s) = 1-\hat{F}(s)$. Now we are ready to state the following important theorem.

\begin{theorem}[Theorem 8.1.6 in \cite{RV}]
\label{thm:RV}
Let $L$ be a slowly varying function, $\mu_n<\infty$, where $n\in\mathbb{Z}^+$, and $\nu = n+\b$ with $0\leq\b\leq1$. Then the following are equivalent:
\begin{itemize}
\item $f_n(s) \sim s^\nu L(1/s)$ as $s\downarrow0$;
\item $1-F(x) \sim \f{(-1)^n}{\Gamma(1-\nu)}x^{-\nu} L(x)$ as $x\to\infty$ when $0<\beta<1$.
\end{itemize}
\end{theorem}

\begin{lemma}[Lemma 9.2 in \cite{Mandjes}]
\label{lem:9.2}
Suppose that for $s\downarrow0$,
\[
\phi'(s) \sim \sum_{i=0}^{n-1} c_i s^i + \eta s^{\nu-1} L(1/s),
\]
for some constants $c_0,\ldots,c_{n-1}$ with $\nu\in(n,n+1)$, and $L$ a slowly varying function. Then for $s\downarrow0$:
\[
\psi(s) \sim \sum_{i=0}^{n} \hat{c}_i s^i - \f\eta\nu \f{1}{(\phi'(0))^{\nu+1}} s^{\nu-1} L(1/s),
\]
for some constants $\hat{c}_0,\ldots,\hat{c}_n$.
\end{lemma}

\end{appendix}

\bibliographystyle{apalike}%Used BibTeX style is unsrt
\bibliography{mybib}

\end{document}